\documentclass{article}%
\usepackage{amsfonts}
\usepackage{graphicx}
\usepackage{amsmath}
\usepackage{amssymb}%
\providecommand{\U}[1]{\protect\rule{.1in}{.1in}}
\newtheorem{theorem}{Theorem}

\newtheorem{corollary}[theorem]{Corollary}

\newtheorem{definition}[theorem]{Definition}

\newtheorem{lemma}[theorem]{Lemma}
\newtheorem{notation}[theorem]{Notation}

\newtheorem{proposition}[theorem]{Proposition}
\newtheorem{remark}[theorem]{Remark}

\newenvironment{proof}[1][Proof]{\textbf{#1.} }{\ \rule{0.5em}{0.5em}}
\begin{document}

\title{The Baker-Campbell-Hausdorff Formula \\and\\the Zassenhaus Formula\\in \\Synthetic DifferentialGeometry}
\author{Hirokazu NISHIMURA\\Institute of Mathematics\\University of Tsukuba\\Tsukuba, Ibaraki, 305-8571, JAPAN}
\maketitle

\begin{abstract}
After the torch of Anders Kock [Taylor series calculus for ring objects of
line type, Journal of Pure and Applied Algebra, 12 (1978), 271-293], we will
establish the Baker-Campbell-Hausdorff formula as well as the Zassenhaus
formula in the theory of Lie groups.

\end{abstract}

\section{Introduction}

The Baker-Campbell-Hausdorff formula (the BCH formula for short) was first
discovered by Campbell (\cite{ca1} and \cite{ca2}) on the closing days of the
19th century so as to construct a Lie group directly from a given Lie algebra
(i.e., Lie's third fundamental theorem !). However, his investigation failed
in convergence problems, let alone dealing only with matrix Lie algebras. The
BCH formula was finally established by Baker \cite{ba} and Hausdorff
\cite{hau} independently within a somewhat more abstract framework of formal
power series on the dawning days of the 20th century, getting rid of
convergence problems completely while losing touch with the theory of Lie
groups. The BCH formula resurrected its touch with the theory of Lie groups
thanks to Magnus \cite{ma0} in the middle of the 20th century.

The BCH formula claims, roughly speaking, that the multiplication in a Lie
group is already encoded in its Lie algebra. More precisely, the
multiplication in a Lie group is expressible in terms of Lie brackets in its
Lie algebra, which readily gives rise to Lie's second fundamental theorem in
the theory of finite-dimensional Lie groups, though the modern treatment of
the theory of finite-dimensional Lie groups is liable to base Lie's second
fundamental theorem somewhat opaquely upon the Frobenius theorem.

The so-called Taylor formula was introduced by the English mathematician
called Brook Taylor in the early 18th century, though its pedigree can be
traced back even to Zeno in ancient Greece. Kock \cite{ko} has shown that the
nature of the Taylor formula in differential calculus is more combinatorial or
algebraic than analytical, dodging convergence problems completely, as far as
we are admitted to speak on the infinitesimal level, where nilpotent
infinitesimals are available in plenty. The principal objective in this paper
is to do the same thing to the BCH formula and its inverse companion called
the Zassenhaus formula in the theory of Lie groups, though we must confront
the \textit{noncommutative} world in sharp contrast to the Taylor formula
living a \textit{commutative} life. We have found out that the Zassenhaus
formula is much easier to deal with than the BCH formula itself, albeit,
historically speaking, the former having been found out by Zassenhaus
\cite{za} within an abstract framework of formal power series more than three
decades later than the latter and its continuous counterpart having been
established by Fer \cite{fe} four years later than \cite{ma0}. Quirky enough,
our BCH formula diverges from the usual one in the $4$-th order. The BCH
formula will be dealt with in \S \ref{s7} and \S \ref{s8} by two different
methods, while we will be engaged upon the Zassenhaus formula in \S \ref{s6}.
We approach the BCH formula in anticipation of its validity in \S \ref{s7} by
using only the left logarithmic derivative of the exponential mapping, while
we will do so from scratch in \S \ref{s8} by using both of the left and right
logarithmic derivatives of the exponential mapping. As is expected, the latter
proofs are longer than the former ones.

We will work within the framework of synthetic differential geometry as in
\cite{la}. We assume the reader to be familiar with Chapters 1-3 of \cite{la}.
Now we fix our terminology and notation. Given a microlinear space $M$, we
denote $M^{D}$\ by $\mathbf{T}M$, while we denote the tangent space of $M$\ at
$x\in M$ by $\mathbf{T}_{x}M=\left\{  \gamma\in\mathbf{T}M\mid\gamma\left(
0\right)  =x\right\}  $. Given a mapping $f:M\rightarrow N$ of microlinear
spaces, its differential is denoted by $\mathbf{d}f$, which is a mapping from
$\mathbf{T}M$ to $\mathbf{T}N$, assigning $f\circ\gamma\in\mathbf{T}N$ to each
$\gamma\in\mathbf{T}M$. We denote the identity mapping of $M$\ by
$\mathrm{id}_{M}$. The unit element of a group $G$\ is usually denoted by $e$.
In the proof of a theorem or the like, we insert some comment surrounded with
parentheses $\left)  {}\right(  $.

\section{\label{s2}The Lie Algebra of a Lie Group}

\begin{definition}
A Lie group is a group which is microlinear as a space.
\end{definition}

\begin{notation}
Given a Lie group $G$, its tangent space $\mathbf{T}_{e}G$\ at $e$\ is usually
denoted by its corresponding German letter $\mathfrak{g}$.
\end{notation}

From now on, $G$\ will always be assumed to be a Lie group with $\mathfrak{g}%
=\mathbf{T}_{e}G$.

\begin{proposition}
\label{t2.1}Given $X\in\mathfrak{g}$\ and $\left(  d_{1},d_{2}\right)  \in
D(2)$, we have
\[
X_{d_{1}+d_{2}}=X_{d_{1}}.X_{d_{2}}%
\]

\end{proposition}

\begin{proof}
By the same token as in Proposition 3, \S 3.2 of \cite{la}.
\end{proof}

\begin{corollary}%
\[
X_{-d}=\left(  X_{d}\right)  ^{-1}%
\]

\end{corollary}

\begin{proof}
Evidently
\[
\left(  d,-d\right)  \in D(2)
\]
obtains, so that we get
\[
e=X_{d+\left(  -d\right)  }=X_{d}.X_{-d}=X_{-d}.X_{d}%
\]
by the above proposition.
\end{proof}

\begin{proposition}
\label{t2.2}Given $X,Y\in\mathfrak{g}$\ and $d\in D$, we have
\[
\left(  X+Y\right)  _{d}=X_{d}.Y_{d}=Y_{d}.X_{d}%
\]

\end{proposition}

\begin{proof}
By the same token as in Proposition 6, \S 3.2 of \cite{la}.
\end{proof}

\begin{theorem}
\label{t2.3}Given $X,Y\in\mathfrak{g}$, there exists a unique $Z\in
\mathfrak{g}$ with
\[
X_{d_{1}}.Y_{d_{2}}.X_{-d_{1}}.Y_{-d_{2}}=Z_{d_{1}d_{2}}%
\]
for any $d_{1},d_{2}\in D$.
\end{theorem}

\begin{proof}
By the same token as in pp.71-72 of \cite{la}.
\end{proof}

\begin{definition}
We denote $Z$\ in the above theorem by $\left[  X,Y\right]  $, so that we have
a function
\[
\left[  \cdot,\cdot\right]  :\mathfrak{g}\times\mathfrak{g}\rightarrow
\mathfrak{g}%
\]
called the Lie bracket.
\end{definition}

\begin{theorem}
\label{t2.4}The $\mathbb{R}$-module $\mathfrak{g}$\ endowed with the Lie
bracket $\left[  \cdot,\cdot\right]  :\mathfrak{g}\times\mathfrak{g}%
\rightarrow\mathfrak{g}$\ is a Lie algebra.
\end{theorem}

\begin{proof}
By the same token as in Proposition 7 (\S 3.2) of \cite{la}.
\end{proof}

\begin{proposition}
Given a homomorphism
\[
\varphi:G\rightarrow H
\]
of Lie groups, the mapping
\[
\varphi^{\prime}:\mathfrak{g}\rightarrow\mathfrak{h}%
\]
obtained as the restriction of the differential
\[
\mathbf{d}\varphi:\mathbf{T}G\rightarrow\mathbf{T}H
\]
to $\mathfrak{g}=\mathbf{T}_{e}G$ is a homomorphism of Lie algebras.
\end{proposition}

\begin{proof}
Given $X,Y\in\mathfrak{g}$ and $d_{1},d_{2}\in D$, we have
\begin{align*}
&  \left(  \mathbf{d}\varphi\left(  \left[  X,Y\right]  \right)  \right)
_{d_{1}d_{2}}\\
&  =\varphi\left(  \left[  X,Y\right]  _{d_{1}d_{2}}\right) \\
&  =\varphi\left(  X_{d_{1}}.Y_{d_{2}}.X_{-d_{1}}.Y_{-d_{2}}\right) \\
&  =\varphi\left(  X_{d_{1}}\right)  .\varphi\left(  Y_{d_{2}}\right)
.\varphi\left(  X_{-d_{1}}\right)  .\varphi\left(  Y_{-d_{2}}\right) \\
&  =\left(  \mathbf{d}\varphi\left(  X\right)  \right)  _{d_{1}}.\left(
\mathbf{d}\varphi\left(  Y\right)  \right)  _{d_{2}}.\left(  \mathbf{d}%
\varphi\left(  X\right)  \right)  _{-d_{1}}.\left(  \mathbf{d}\varphi\left(
Y\right)  \right)  _{-d_{2}}\\
&  =\left[  \mathbf{d}\varphi\left(  X\right)  ,\mathbf{d}\varphi\left(
Y\right)  \right]  _{d_{1}d_{2}}%
\end{align*}
so that $\varphi^{\prime}$\ preserves Lie brackets.
\end{proof}

The succeeding simple lemma will be useful in the last section.

\begin{lemma}
\label{t2.5}Given $X,Y\in\mathfrak{g}$, we have
\[
\left[  X,\left[  Y,\left[  X,Y\right]  \right]  \right]  =\left[  Y,\left[
X,\left[  X,Y\right]  \right]  \right]
\]

\end{lemma}

\begin{proof}
This follows easily from the following Jacobi identity:
\[
\left[  X,\left[  Y,\left[  X,Y\right]  \right]  \right]  +\left[  Y,\left[
\left[  X,Y\right]  ,X\right]  \right]  +\left[  \left[  X,Y\right]  ,\left[
X,Y\right]  \right]  =0
\]

\end{proof}

\begin{notation}
Given a Euclidean $\mathbb{R}$-module $V$\ which is microlinear as a space,
the totality of bijective homomorphisms of $\mathbb{R}$-modules from $V$\ onto
itself\ is denoted by $GL\left(  V\right)  $, which is a Lie group with
composition of mappings as its group operation (cf. Proposition 5 (\S \S 3.2)
of \cite{la}). Its Lie algebra is usually denoted by $\mathfrak{gl}\left(
V\right)  $.
\end{notation}

\begin{proposition}
\label{t2.6}Given a Euclidean $\mathbb{R}$-module $V$\ which is microlinear as
a space, the Lie algebra $\mathfrak{gl}\left(  V\right)  $\ can naturally be
identified with the Lie algebra of homomorphisms of $\mathbb{R}$-modules from
$V$\ into itself with its Lie bracket
\[
\left[  \varphi,\psi\right]  =\varphi\circ\psi-\psi\circ\varphi
\]
for any homomorphisms $\varphi,\psi$\ of $\mathbb{R}$-modules from $V$\ into itself.
\end{proposition}

\begin{proof}
Given a mapping $X:D\rightarrow GL\left(  V\right)  $ with $X_{0}%
=\mathrm{id}_{V}$, there exists a unique mapping $\varphi:V\rightarrow V$ such
that
\[
X_{d}\left(  u\right)  =u+d\varphi\left(  u\right)
\]
for any $d\in D$ and any $u\in V$, since the $\mathbb{R}$-module $V$\ is
Euclidean by assumption. Since $X_{d}\in GL\left(  V\right)  $, we have
\[
\alpha u+d\varphi\left(  \alpha u\right)  =X_{d}\left(  \alpha u\right)
=\alpha X_{d}\left(  u\right)  =\alpha u+\alpha d\varphi\left(  u\right)
\]
for any $\alpha\in\mathbb{R}$, any $u\in V$ and any $d\in D$, so that we get
\[
\varphi\left(  \alpha u\right)  =\alpha\varphi\left(  u\right)
\]
for any $\alpha\in\mathbb{R}$ and any $u\in V$, which implies that the mapping
$\varphi:V\rightarrow V$\ is a homomorphism of $\mathbb{R}$-modules (cf.
Proposition 10 (\S \S 1.2) in \cite{la}). Conversely, given a homomorphism
$\varphi$\ of $\mathbb{R}$-modules from $V$\ into itself and $d\in D$,
$\mathrm{id}_{V}+d\varphi$\ is obviously a homomorphism of $\mathbb{R}%
$-modules from $V$\ into itself, and we have
\[
\left(  \mathrm{id}_{V}+d\varphi\right)  \circ\left(  \mathrm{id}_{V}%
-d\varphi\right)  =\left(  \mathrm{id}_{V}-d\varphi\right)  \circ\left(
\mathrm{id}_{V}+d\varphi\right)  =\mathrm{id}_{V}%
\]
so that the mapping $\mathrm{id}_{V}+d\varphi$\ is bijective. Therefore we are
sure that the $\mathbb{R}$-module $\mathfrak{gl}\left(  V\right)  $ is
naturally identified with the $\mathbb{R}$-module of homomorphisms of
$\mathbb{R}$-modules from $V$\ into itself. It remains to show that this
identification preserves Lie brackets. Let us assume that $X\in\mathfrak{gl}%
\left(  V\right)  $ corresponds to the homomorphism $\varphi$\ of $\mathbb{R}%
$-modules from $V$\ into itself, while $Y\in\mathfrak{gl}\left(  V\right)  $
corresponds to the homomorphism $\psi$\ of $\mathbb{R}$-modules from $V$\ into
itself. Then, given $d_{1},d_{2}\in D$, we have
\begin{align*}
&  \left[  X,Y\right]  _{d_{1}d_{2}}\\
&  =X_{d_{1}}.Y_{d_{2}}.X_{-d_{1}}.Y_{-d_{2}}\\
&  =\left(  \mathrm{id}_{V}+d_{1}\varphi\right)  \circ\left(  \mathrm{id}%
_{V}+d_{2}\psi\right)  \circ\left(  \mathrm{id}_{V}-d_{1}\varphi\right)
\circ\left(  \mathrm{id}_{V}-d_{2}\psi\right) \\
&  =\left\{  \mathrm{id}_{V}+d_{1}\varphi+d_{2}\psi+d_{1}d_{2}\varphi\circ
\psi\right\}  \circ\left\{  \mathrm{id}_{V}-d_{1}\varphi-d_{2}\psi+d_{1}%
d_{2}\varphi\circ\psi\right\} \\
&  =\mathrm{id}_{V}-d_{1}\varphi-d_{2}\psi+d_{1}d_{2}\varphi\circ\psi
+d_{1}\varphi-d_{1}d_{2}\varphi\circ\psi+d_{2}\psi-d_{1}d_{2}\psi\circ
\varphi+d_{1}d_{2}\varphi\circ\psi\\
&  =\mathrm{id}_{V}+d_{1}d_{2}\left(  \varphi\circ\psi-\psi\circ
\varphi\right)
\end{align*}
so that our identification of $\mathfrak{gl}\left(  V\right)  $ with the
$\mathbb{R}$-module of homomorphisms of $\mathbb{R}$-modules from $V$\ into
itself indeed preserves Lie brackets.
\end{proof}

\section{\label{s3}The Adjoint Representations}

\begin{notation}
Given $x\in G$, the mapping $y\in G\mapsto xyx^{-1}\in G$ is obviously a
homomorphism of groups, naturally giving rise to a mapping $\mathfrak{g}%
\rightarrow\mathfrak{g}$ as derivation, which we denote by $\mathrm{Ad\,}x\in
GL\left(  \mathfrak{g}\right)  $. Thus we have a homomorphism $\mathrm{Ad}%
:G\rightarrow GL\left(  \mathfrak{g}\right)  $ of groups, naturally giving
rise to a mapping $\mathrm{ad}:\mathfrak{g}\rightarrow\mathfrak{gl}\left(
\mathfrak{g}\right)  $ as derivation.
\end{notation}

\begin{theorem}
Given $X,Y\in\mathfrak{g}$, we have
\[
\left(  \mathrm{ad}\,X\right)  \left(  Y\right)  =\left[  X,Y\right]
\]

\end{theorem}

\begin{proof}
Given $d,d^{\prime}\in D$, we have
\begin{align*}
&  \left(  \left(  \mathrm{Ad\,}X_{d}\right)  \left(  Y\right)  -Y\right)
_{d^{\prime}}\\
&  =X_{d}.Y_{d^{\prime}}.X_{-d}.Y_{-d^{\prime}}\\
&  \left)  \text{By Proposition \ref{t2.2}}\right( \\
&  =\left[  X,Y\right]  _{dd^{\prime}}\\
&  =\left(  d\left[  X,Y\right]  \right)  _{d^{\prime}}%
\end{align*}
so that we have the desired formula.
\end{proof}

\section{\label{s4}The Exponential Mapping}

Our notions of a one-parameter subgroup, a left-invariant vector field, etc.
are standard, and it is easy to see that

\begin{proposition}
\label{t4.1}Given a mapping $\theta:\mathbb{R}\rightarrow G$, the following
conditions are equivalent:

\begin{enumerate}
\item The mapping $\theta:\mathbb{R}\rightarrow G$\ is a one-parameter subgroup.

\item The mapping $\theta:\mathbb{R}\rightarrow G$\ is a flow of a left
invariant vector field on $G$ with $\theta\left(  0\right)  =e$.

\item The mapping $\theta:\mathbb{R}\rightarrow G$\ is a flow of a right
invariant vector field on $G$ with $\theta\left(  0\right)  =e$.
\end{enumerate}
\end{proposition}

\begin{notation}
Given $X\in\mathfrak{g}$, if there is a one-parameter subgroup $\theta
:\mathbb{R}\rightarrow G$ with $\mathbf{d}\theta\left(  i_{D}^{\mathbb{R}%
}\right)  =X$, then we write $\exp^{G}\,X$ or $\exp\,X$ for $\theta\left(
1\right)  $.
\end{notation}

The following definition is borrowed from 38.4 in \cite{krmi}, which is, in
turn, owing to the research \cite{omy1}-\cite{omy6} of Omori et al.

\begin{definition}
A Lie group $G$\ is called regular provided that, for any mapping
$\varsigma:\mathbb{R}\rightarrow\mathfrak{g}$, there exists a mapping
$\theta:\mathbb{R}\rightarrow G$ with
\[
\theta\left(  0\right)  =e
\]
and
\[
\theta(t+d)=\theta\left(  t\right)  .\varsigma\left(  t\right)  _{d}%
\]
for any $t\in\mathbb{R}$ and any $d\in D$.
\end{definition}

From now on, we will assume the Lie group $G$ to be regular, so that $\exp
^{G}:\mathfrak{g}\rightarrow G$ is indeed a total function.

\begin{notation}
Given $\xi\in\mathfrak{gl}\left(  V\right)  $ with $\xi^{n+1}$ vanishing for
some natural number $n$, we write
\[
e^{\xi}=\sum_{i=0}^{n}\frac{\xi^{i}}{i!}%
\]

\end{notation}

It is easy to see that

\begin{lemma}
\label{t4.2}Given $\xi\in\mathfrak{gl}\left(  V\right)  $ with $\xi^{n+1}$
vanishing for some natural number $n$, we have
\[
\exp^{GL\left(  V\right)  }\,\xi=e^{\xi}%
\]

\end{lemma}

\begin{proposition}
\label{t4.3}Given a homomorphism $\varphi:G\rightarrow H$ of Lie groups and
$X\in\mathfrak{g}$, $\exp^{H}$\thinspace$\varphi^{\prime}\left(  X\right)  $
is defined, and we have
\[
\exp^{H}\,\varphi^{\prime}\left(  X\right)  =\varphi\left(  \exp
^{G}\,X\right)
\]

\end{proposition}

\begin{remark}
The Lie group $G$\ is assumed to be regular, as we have said before, but the
Lie group $H$\ is not assumed to be regular, so that $\exp^{H}$ is not
necessarily a total function.
\end{remark}

\begin{proof}
It suffices to note that, given a one-parameter subgroup $\theta
:\mathbb{R}\rightarrow G$\ of $G$\ with
\[
\mathbf{d}\theta\left(  i_{D}^{\mathbb{R}}\right)  =X\text{,}%
\]
the mapping $\varphi\circ\theta:\mathbb{R}\rightarrow H$\ is a one-parameter
subgroup of $H$\ with
\[
\mathbf{d}\left(  \varphi\circ\theta\right)  \left(  i_{D}^{\mathbb{R}%
}\right)  =\varphi^{\prime}\left(  X\right)
\]

\end{proof}

\begin{proposition}
\label{t4.4}Given $X\in\mathfrak{g}$\ with $\left(  \mathrm{ad}\,X\right)
^{n+1}$ vanishing for some natural number $n$, we have
\[
\mathrm{Ad\,}\left(  \exp\,X\right)  =e^{\mathrm{ad}\,X}%
\]

\end{proposition}

\begin{proof}
We have
\begin{align*}
&  \mathrm{Ad\,}\left(  \exp^{G}\,X\right) \\
&  =\exp^{GL\left(  V\right)  }\,\left(  \mathrm{ad}\,X\right) \\
&  \left)  \text{By Proposition \ref{t4.3}}\right( \\
&  =e^{\mathrm{ad}\,X}\\
&  \left)  \text{By Lemma \ref{t4.2}}\right(
\end{align*}

\end{proof}

We conclude this section by the following simple but significant proposition.

\begin{proposition}
\label{t4.5}We have
\[
\exp\,t\left(  dX\right)  =X_{td}%
\]
for any $t\in\mathbb{R}$. In particular, we have
\[
\exp\,dX=X_{d}%
\]
by setting
\[
t=1
\]

\end{proposition}

\begin{proof}
For any $d^{\prime}\in D$, we have
\begin{align*}
&  \left(  dX\right)  _{t+d^{\prime}}\\
&  =X_{\left(  t+d^{\prime}\right)  d}\\
&  =X_{td+d^{\prime}d}\\
&  =X_{td}.X_{d^{\prime}d}\\
&  \left)  \text{By Proposition \ref{t2.1}}\right( \\
&  =\left(  dX\right)  _{t}.\left(  dX\right)  _{d^{\prime}}%
\end{align*}
so that we have the desired conclusion.
\end{proof}

\section{\label{s5}Logarithmic Derivatives}

In this section we deal with the left and right derivations. First we deal
with the left derivation.

\begin{definition}
Given a microlinear space $M$\ and a function $f:M\rightarrow G$, the
function
\[
\delta^{\mathrm{left}}f:\mathbf{T}M\rightarrow\mathfrak{g}%
\]
is defined to be such that
\[
\left(  \mathbf{d}f\left(  X\right)  \right)  _{d}=f\left(  x\right)  .\left(
\delta^{\mathrm{left}}f\left(  X\right)  \right)  _{d}%
\]
for any $x\in M$, any $X\in\mathbf{T}_{x}M$ and any $d\in D$. It is called the
left logarithmic derivative of $f$. The restriction of $\delta^{\mathrm{left}%
}f$ to $\mathbf{T}_{x}M$\ is denoted by $\delta^{\mathrm{left}}f\left(
x\right)  $.
\end{definition}

The following is the Leibniz rule for the left logarithmic derivation.

\begin{proposition}
\label{t5.1}Let $M$\ be a microlinear space. Given two functions
\[
f,g:M\rightarrow G
\]
together with $X\in\mathbf{T}M$, we have
\[
\delta^{\mathrm{left}}\left(  fg\right)  \left(  X\right)  =\delta
^{\mathrm{left}}g\left(  X\right)  +\mathrm{Ad\,}\left(  g\left(  x\right)
^{-1}\right)  \left(  \delta^{\mathrm{left}}f\left(  X\right)  \right)
\]
with
\[
x=X_{0}%
\]

\end{proposition}

\begin{proof}
For any $d\in D$, we have
\begin{align*}
&  \left(  \delta^{\mathrm{left}}\left(  fg\right)  \left(  X\right)  \right)
_{d}\\
&  =g(x)^{-1}.f\left(  x\right)  ^{-1}.f\left(  X_{d}\right)  .g\left(
X_{d}\right) \\
&  =g(x)^{-1}.f\left(  x\right)  ^{-1}.f\left(  X_{d}\right)  .g(x).g(x)^{-1}%
.g\left(  X_{d}\right) \\
&  =\left\{  \mathrm{Ad\,}\left(  g\left(  x\right)  ^{-1}\right)  \left(
\delta^{\mathrm{left}}f\left(  X\right)  \right)  +\delta g\left(  X\right)
\right\}  _{d}\\
&  \left)  \text{By Proposition \ref{t2.2}}\right(
\end{align*}
so that we get the desired formula.
\end{proof}

\begin{theorem}
\label{t5.2}Given $X\in\mathfrak{g}$\ with $\left(  \mathrm{ad}\,X\right)
^{n+1}$ vanishing for some natural number $n$, we have
\[
\delta^{\mathrm{left}}\left(  \exp\right)  \left(  X\right)  =\sum_{p=0}%
^{n}\frac{\left(  -1\right)  ^{p}}{\left(  p+1\right)  !}\left(
\mathrm{ad}\,X\right)  ^{p}%
\]

\end{theorem}

\begin{proof}
The proof is essentially on the lines of Lemma 4.27 of \cite{mi}. We have
\begin{align*}
&  \left(  s+t\right)  \delta^{\mathrm{left}}\left(  \exp\right)  \left(
\left(  s+t\right)  X\right) \\
&  =\delta^{\mathrm{left}}\left(  \exp\,\left(  s+t\right)  \cdot\right)
\left(  X\right) \\
&  \text{[By the chain rule of differentiation]}\\
&  =\delta^{\mathrm{left}}\left(  \left(  \exp\,s\cdot\right)  \left(
\exp\,t\cdot\right)  \right)  \left(  X\right) \\
&  =\delta^{\mathrm{left}}\left(  \exp\,t\cdot\right)  \left(  X\right)
+\mathrm{Ad\,}\left(  \exp\,\left(  -t\right)  X\right)  \left(
\delta^{\mathrm{left}}\left(  \exp\,s\cdot\right)  \right)  \left(  X\right)
\\
&  \text{[By Proposition \ref{t5.1}]}\\
&  =t\delta^{\mathrm{left}}\left(  \exp\right)  \left(  tX\right)
+\mathrm{Ad\,}\left(  \exp\,\left(  -t\right)  X\right)  \left(
s\delta^{\mathrm{left}}\left(  \exp\right)  \left(  sX\right)  \right)
\end{align*}
so that, by letting
\[
F\left(  s\right)  =s\delta^{\mathrm{left}}\left(  \exp\right)  \left(
sX\right)
\]
so as to introduce a function
\[
F:\mathbb{R}\rightarrow L\left(  \mathfrak{g},\mathfrak{g}\right)  \text{,}%
\]
we get
\[
F(s+t)=F\left(  t\right)  +\mathrm{Ad\,}\left(  \exp\,\left(  -t\right)
X\right)  \left(  F\left(  s\right)  \right)  \text{,}%
\]
which earns us
\begin{equation}
F^{\prime}\left(  s\right)  =F^{\prime}\left(  0\right)  -\left(
\mathrm{ad}\,X\right)  \left(  F\left(  s\right)  \right)  \label{5.2.1}%
\end{equation}
by fixing $s$\ and differentiaing with respect to $t$\ at $t=0$. Since we have
also
\[
F^{\prime}\left(  s\right)  =\delta^{\mathrm{left}}\left(  \exp\right)
\left(  sX\right)  +s\delta^{\mathrm{left}}\left(  \exp\right)  \left(
X\right)  \text{,}%
\]
we get
\[
F^{\prime}\left(  0\right)  =\mathrm{id}_{\mathfrak{g}}\text{,}%
\]
by letting $s=0$,\ so that the formula (\ref{5.2.1}) is transmogrified into
the ordinary differential equation
\[
F^{\prime}\left(  s\right)  =\mathrm{id}_{\mathfrak{g}}-\left(  \mathrm{ad}%
\,X\right)  \left(  F\left(  s\right)  \right)
\]
on $L\left(  \mathfrak{g},\mathfrak{g}\right)  $. Its unique solution with the
initial condition of $F\left(  0\right)  $'s vanishing is
\[
F\left(  s\right)  =\sum_{p=0}^{n}\frac{\left(  -1\right)  ^{p}s^{p+1}%
}{\left(  p+1\right)  !}\left(  \mathrm{ad}\,X\right)  ^{p}\text{,}%
\]
which results in the desired formula by letting $s=1$.
\end{proof}

\begin{proposition}
\label{t5.3}Given $X,Y\in\mathfrak{g}$\ with $\left[  X,Y\right]
$\ vanishing, we have
\[
\exp\,X.\exp\,Y=\exp\,X+Y
\]
In particular, we have
\[
\exp\,X.\exp\,Y=\exp\,Y.\exp\,X
\]

\end{proposition}

\begin{proof}
Letting
\[
H\left(  t\right)  =\exp\,X.\exp\,tY.\exp\,-\left(  X+tY\right)
\]
so as to get a function
\[
H:\mathbb{R}\rightarrow G\text{,}%
\]
we have
\[
H\left(  0\right)  =e
\]
evidently. By differentiating $H$\ logarithmically, we have
\begin{align*}
&  \delta^{\mathrm{left}}H\left(  t\right) \\
&  =\delta^{\mathrm{left}}\left(  \exp\right)  \left(  -\left(  X+tY\right)
\right)  \left(  -Y\right)  +\mathrm{Ad\,}\left(  \exp\,X+tY\right)  \left(
\delta^{\mathrm{left}}\left(  \exp\right)  \left(  tY\right)  \left(
Y\right)  \right) \\
&  =-Y+\mathrm{Ad\,}\left(  \exp\,X+tY\right)  \left(  Y\right) \\
&  =-Y+e^{\mathrm{ad\,}\left(  \,X+tY\right)  }\left(  Y\right) \\
&  =-Y+Y\\
&  =0
\end{align*}
so that we have the desired formula.
\end{proof}

\begin{proposition}
\label{t5.4}Given $X,Y\in\mathfrak{g}$\ and $d_{1},d_{2}\in D$, we have
\begin{align*}
&  \exp\,d_{1}X.\exp\,d_{2}Y\\
&  =\exp\,d_{2}Y.\exp\,d_{1}X.\exp\,d_{1}d_{2}\left[  X,Y\right]
\end{align*}

\end{proposition}

\begin{proof}
we have
\begin{align*}
&  \exp\,d_{1}X+d_{2}Y\\
&  =\exp\,d_{1}X.\left\{  \delta^{\mathrm{left}}\left(  \exp\right)  \left(
d_{1}X\right)  \left(  Y\right)  \right\}  _{d_{2}}\\
&  \left)  \text{logarithmic derivation}\right( \\
&  =\exp\,d_{1}X.\left\{  Y-\frac{1}{2}d_{1}\left[  X,Y\right]  \right\}
_{d_{2}}\\
&  \left)  \text{By Theorem \ref{t5.2}}\right( \\
&  =\exp\,d_{1}X.Y_{d_{2}}.\left(  -\frac{1}{2}d_{1}\left[  X,Y\right]
\right)  _{d_{2}}\\
&  \left)  \text{By Proposition \ref{t2.2}}\right( \\
&  =\exp\,d_{1}X.\exp\,d_{2}Y.\exp\,-\frac{1}{2}d_{1}d_{2}\left[  X,Y\right]
\\
&  \left)  \text{By Proposition \ref{t4.5}}\right(  \text{,}%
\end{align*}
while we have
\begin{align*}
&  \exp\,d_{1}X+d_{2}Y\\
&  =\exp\,d_{2}Y+d_{1}X\\
&  =\exp\,d_{2}Y.\exp\,d_{1}X.\exp\,-\frac{1}{2}d_{1}d_{2}\left[  Y,X\right]
\end{align*}
by the same token. Therefore we have
\begin{align*}
&  \exp\,d_{1}X.\exp\,d_{2}Y.\exp\,-\frac{1}{2}d_{1}d_{2}\left[  X,Y\right] \\
&  =\exp\,d_{2}Y.\exp\,d_{1}X.\exp\,-\frac{1}{2}d_{1}d_{2}\left[  Y,X\right]
\end{align*}
By multiplying
\[
\exp\,\frac{1}{2}d_{1}d_{2}\left[  X,Y\right]
\]
from the right and making use of Proposition \ref{t5.3}, we get the desired formula.
\end{proof}

Now we deal with the right derivation.

\begin{definition}
Given a microlinear space $M$\ and a function $f:M\rightarrow G$, the
function
\[
\delta^{\mathrm{right}}f:\mathbf{T}M\rightarrow\mathfrak{g}%
\]
is defined to be such that
\[
\left(  \mathbf{d}f\left(  X\right)  \right)  _{d}=\left(  \delta
^{\mathrm{right}}f\left(  X\right)  \right)  _{d}.f\left(  x\right)
\]
for any $x\in M$, any $X\in\mathbf{T}_{x}M$ and any $d\in D$. It is called the
right logarithmic derivative of $f$. The restriction of $\delta
^{\mathrm{right}}f$ to $\mathbf{T}_{x}M$\ is denoted by $\delta
^{\mathrm{right}}f\left(  x\right)  $.
\end{definition}

\begin{proposition}
\label{t5.5}Let $M$\ be a microlinear space. Given two functions
\[
f,g:M\rightarrow G
\]
together with $X\in\mathbf{T}M$, we have
\[
\delta^{\mathrm{right}}\left(  fg\right)  \left(  X\right)  =\delta
^{\mathrm{right}}f\left(  X\right)  +\mathrm{Ad\,}\left(  f(x)\right)  \left(
\delta^{\mathrm{right}}g\left(  X\right)  \right)
\]
with
\[
x=X_{0}%
\]

\end{proposition}

\begin{theorem}
\label{t5.6}Given $X\in\mathfrak{g}$\ with $\left(  \mathrm{ad}\,X\right)
^{n+1}$ vanishing for some natural number $n$, we have
\[
\delta^{\mathrm{right}}\left(  \exp\right)  \left(  X\right)  =\sum_{p=0}%
^{n}\frac{1}{\left(  p+1\right)  !}\left(  \mathrm{ad}\,X\right)  ^{p}%
\]

\end{theorem}

\section{\label{s6}The Zassenhaus Formula}

\begin{lemma}
\label{t6.0}Given $d_{1},...d_{n}\in D$, we have
\[
\frac{\left(  d_{1}+...+d_{n}\right)  ^{m}}{m!}=\sum_{i_{1}<...<i_{m}}%
d_{i_{1}}...d_{i_{m}}%
\]
for any natural number $m$\ with $m\leq n$.
\end{lemma}

\begin{proof}
The reader is referred to Lemma (p.10) of \cite{la}.
\end{proof}

\begin{theorem}
\label{t6.1}Given $X,Y\in\mathfrak{g}$\ and $d_{1}\in D$, we have
\begin{align*}
&  \exp\,d_{1}\left(  X+Y\right) \\
&  =\exp\,d_{1}X.\exp\,d_{1}Y
\end{align*}

\end{theorem}

\begin{proof}
We have
\begin{align*}
&  \exp d_{1}\,\left(  X+Y\right) \\
&  =\left(  X+Y\right)  _{d_{1}}\\
&  \left)  \text{By Proposition \ref{t4.5}}\right( \\
&  =X_{d_{1}}.Y_{d_{1}}\\
&  \left)  \text{By Proposition \ref{t2.2}}\right( \\
&  =\exp\,d_{1}X.\exp\,d_{1}Y\\
&  \left)  \text{By Proposition \ref{t4.5}}\right(
\end{align*}
so that we have got to the desired formula.
\end{proof}

\begin{theorem}
\label{t6.2}Given $X,Y\in\mathfrak{g}$\ and $d_{1},d_{2}\in D$, we have
\begin{align*}
&  \exp\,\left(  d_{1}+d_{2}\right)  \left(  X+Y\right) \\
&  =\exp\,\left(  d_{1}+d_{2}\right)  X.\exp\,\left(  d_{1}+d_{2}\right)
Y.\exp\,-d_{1}d_{2}\left[  X,Y\right] \\
&  =\exp\,\left(  d_{1}+d_{2}\right)  X.\exp\,\left(  d_{1}+d_{2}\right)
Y.\exp\,-\frac{\left(  d_{1}+d_{2}\right)  ^{2}}{2}\left[  X,Y\right]
\end{align*}

\end{theorem}

\begin{proof}
We have
\begin{align*}
&  \exp\,\left(  d_{1}+d_{2}\right)  \left(  X+Y\right) \\
&  =\exp\,d_{1}\left(  X+Y\right)  +d_{2}\left(  X+Y\right) \\
&  =\exp\,d_{1}\left(  X+Y\right)  .\left\{  \delta^{\mathrm{left}}\left(
\exp\right)  \left(  d_{1}\left(  X+Y\right)  \right)  \left(  X+Y\right)
\right\}  _{d_{2}}\\
&  \left)  \text{left logarithmic derivation}\right( \\
&  =\exp\,d_{1}\left(  X+Y\right)  .\left(  X+Y\right)  _{d_{2}}\\
&  \left)  \text{By Theorem \ref{t5.2}}\right( \\
&  =\exp\,d_{1}\left(  X+Y\right)  .\exp\,d_{2}\left(  X+Y\right) \\
&  \text{[By Proposition \ref{t4.5}]}\\
&  =\exp\,d_{1}X.\exp\,d_{1}Y.\exp\,d_{2}X.\exp\,d_{2}Y\\
&  \left)  \text{By Theorem \ref{t6.1}}\right( \\
&  =\exp\,d_{1}X.\exp\,d_{2}X.\exp\,d_{1}Y.\exp\,d_{1}d_{2}\left[  Y,X\right]
.\exp\,d_{2}Y\\
&  \left)  \text{By Proposition \ref{t5.4}}\right( \\
&  =\exp\,d_{1}X.\exp\,d_{2}X.\exp\,d_{1}Y.\exp\,d_{2}Y.\exp\,d_{1}d_{2}
\left[  Y,X\right] \\
&  \left)  \text{By Proposition \ref{t5.3}}\right( \\
&  =\exp\,\left(  d_{1}+d_{2}\right)  X.\exp\,\left(  d_{1}+d_{2}\right)
Y.\exp\,d_{1}d_{2}\left[  Y,X\right] \\
&  \left)  \text{By Proposition \ref{t5.3}}\right( \\
&  =\exp\,\left(  d_{1}+d_{2}\right)  X.\exp\,\left(  d_{1}+d_{2}\right)
Y.\exp\,-d_{1}d_{2}\left[  X,Y\right]
\end{align*}
so that we have got to the desired formula.
\end{proof}

\begin{theorem}
\label{t6.3}Given $X,Y\in\mathfrak{g}$\ and $d_{1},d_{2},d_{3}\in D$, we have
\begin{align*}
&  \exp\,\left(  d_{1}+d_{2}+d_{3}\right)  \left(  X+Y\right) \\
&  =\exp\,\left(  d_{1}+d_{2}+d_{3}\right)  X.\exp\,\left(  d_{1}+d_{2}%
+d_{3}\right)  Y.\exp\,-\left(  d_{1}d_{2}+d_{1}d_{3}+d_{2}d_{3}\right)
\left[  X,Y\right]  .\\
&  \exp\,d_{1}d_{2}d_{3}\left[  X+2Y,\left[  X,Y\right]  \right] \\
&  =\exp\,\left(  d_{1}+d_{2}+d_{3}\right)  X.\exp\,\left(  d_{1}+d_{2}%
+d_{3}\right)  Y.\exp\,-\frac{\left(  d_{1}+d_{2}+d_{3}\right)  ^{2}}%
{2}\left[  X,Y\right]  .\\
&  \exp\,\frac{\left(  d_{1}+d_{2}+d_{3}\right)  ^{3}}{12}\left[  X+2Y,\left[
X,Y\right]  \right]
\end{align*}

\end{theorem}

\begin{proof}
We have
\begin{align*}
&  \exp\,\left(  d_{1}+d_{2}+d_{3}\right)  \left(  X+Y\right) \\
&  =\exp\,\left(  d_{1}+d_{2}\right)  \left(  X+Y\right)  +d_{3}\left(
X+Y\right) \\
&  =\exp\,\left(  d_{1}+d_{2}\right)  \left(  X+Y\right)  .\left\{
\delta^{\mathrm{left}}\left(  \exp\right)  \left(  \left(  d_{1}+d_{2}\right)
\left(  X+Y\right)  \right)  \left(  X+Y\right)  \right\}  _{d_{3}}\\
&  \left)  \text{left logarithmic derivation}\right( \\
&  =\exp\,\left(  d_{1}+d_{2}\right)  \left(  X+Y\right)  .\left(  X+Y\right)
_{d_{3}}\\
&  \left)  \text{By Theorem \ref{t5.2}}\right( \\
&  =\exp\,\left(  d_{1}+d_{2}\right)  \left(  X+Y\right)  .\exp\,d_{3}\left(
X+Y\right) \\
&  \left)  \text{By Proposition \ref{t4.5}}\right( \\
&  =\exp\,\left(  d_{1}+d_{2}\right)  X.\exp\,\left(  d_{1}+d_{2}\right)
Y.\exp\,-d_{1}d_{2}\left[  X,Y\right]  .\exp\,d_{3}X.\exp\,d_{3}Y\\
&  \left)  \text{By Theorems \ref{t6.1} and \ref{t6.2}}\right( \\
&  =\exp\,\left(  d_{1}+d_{2}\right)  X.\exp\,d_{1}Y.\exp\,d_{2}Y.\exp
\,-d_{1}d_{2}\left[  X,Y\right]  .\exp\,d_{3}X.\exp\,d_{3}Y\\
&  =\exp\,\left(  d_{1}+d_{2}+d_{3}\right)  X.\exp\,d_{1}Y.\exp\,d_{1}d_{3}
\left[  Y,X\right]  .\exp\,d_{2}Y.\exp\,d_{2}d_{3}\left[  Y,X\right]  .\\
&  \exp\,-d_{1}d_{2}\left[  X,Y\right]  .\exp\,-d_{1}d_{2}d_{3}\left[  \left[
X,Y\right]  ,X\right]  .\exp\,d_{3}Y\\
&  \left)
\begin{array}
[c]{c}%
\text{By moving }\exp\,d_{3}X\text{\ left towards }\exp\,\left(  d_{1}%
+d_{2}\right)  X\text{ via Propositions \ref{t5.3} and }\\
\text{\ref{t5.4}}%
\end{array}
\right( \\
&  =\exp\,\left(  d_{1}+d_{2}+d_{3}\right)  X.\exp\,\left(  d_{1}%
+d_{2}\right)  Y.\exp\,d_{1}d_{3}\left[  Y,X\right]  .\exp\,d_{1}d_{2}%
d_{3}\left[  \left[  Y,X\right]  ,Y\right]  .\\
&  \exp\,d_{2}d_{3}\left[  Y,X\right]  .\exp\,-d_{1}d_{2}\left[  X,Y\right]
.\exp\,-d_{1}d_{2}d_{3}\left[  \left[  X,Y\right]  ,X\right]  .\exp\,d_{3}Y\\
&  \left)  \text{By exchanging }\exp\,d_{1}d_{3}\left[  Y,X\right]
\text{\ and }\exp\,d_{2}Y\text{ via Proposition \ref{t5.4}}\right( \\
&  =\exp\,\left(  d_{1}+d_{2}+d_{3}\right)  X.\exp\,\left(  d_{1}+d_{2}%
+d_{3}\right)  Y.\exp\,d_{1}d_{3}\left[  Y,X\right]  .\exp\,d_{1}d_{2}%
d_{3}\left[  \left[  Y,X\right]  ,Y\right]  .\\
&  \exp\,d_{2}d_{3}\left[  Y,X\right]  .\exp\,-d_{1}d_{2}\left[  X,Y\right]
.\exp\,-d_{1}d_{2}d_{3}\left[  \left[  X,Y\right]  ,Y\right]  .\exp
\,-d_{1}d_{2}d_{3}\left[  \left[  X,Y\right]  ,X\right] \\
&  \left)
\begin{array}
[c]{c}%
\text{By moving }\exp\,d_{3}Y\text{\ left towards }\exp\,\left(  d_{1}%
+d_{2}\right)  Y\text{ via Propositions \ref{t5.3} and }\\
\text{\ref{t5.4}}%
\end{array}
\right( \\
&  =\exp\,\left(  d_{1}+d_{2}+d_{3}\right)  X.\exp\,\left(  d_{1}+d_{2}%
+d_{3}\right)  Y.\exp\,-\left(  d_{1}d_{2}+d_{1}d_{3}+d_{2}d_{3}\right)
\left[  X,Y\right]  .\\
&  \exp\,d_{1}d_{2}d_{3}\left[  X+2Y,\left[  X,Y\right]  \right]
\end{align*}
so that we have got to the desired formula.
\end{proof}

\begin{theorem}
\label{t6.4}Given $X,Y\in\mathfrak{g}$\ and $d_{1},d_{2},d_{3},d_{4}\in D$, we
have
\begin{align*}
&  \exp\,\left(  d_{1}+d_{2}+d_{3}+d_{4}\right)  \left(  X+Y\right) \\
&  =\exp\,\left(  d_{1}+d_{2}+d_{3}+d_{4}\right)  X.\exp\,\left(  d_{1}%
+d_{2}+d_{3}+d_{4}\right)  Y.\\
&  \exp\,-\left(  d_{1}d_{2}+d_{1}d_{3}+d_{1}d_{4}+d_{2}d_{3}+d_{2}d_{4}%
+d_{3}d_{4}\right)  \left[  X,Y\right]  .\\
&  \exp\,\left(  d_{1}d_{2}d_{3}+d_{1}d_{2}d_{4}+d_{1}d_{3}d_{4}+d_{2}%
d_{3}d_{4}\right)  \left[  X+2Y,\left[  X,Y\right]  \right]  .\\
&  \exp\,d_{1}d_{2}d_{3}d_{4}\left(  -\left[  X,\left[  X,\left[  X,Y\right]
\right]  \right]  -3\left[  X,\left[  Y,\left[  X,Y\right]  \right]  \right]
-3\left[  Y,\left[  Y,\left[  X,Y\right]  \right]  \right]  \right) \\
&  =\exp\,\left(  d_{1}+d_{2}+d_{3}+d_{4}\right)  X.\exp\,\left(  d_{1}%
+d_{2}+d_{3}+d_{4}\right)  Y.\\
&  \exp\,-\frac{\left(  d_{1}+d_{2}+d_{3}+d_{4}\right)  ^{2}}{2}\left[
X,Y\right]  .\\
&  \exp\,\frac{\left(  d_{1}+d_{2}+d_{3}+d_{4}\right)  ^{3}}{12}\left[
X+2Y,\left[  X,Y\right]  \right]  .\\
&  \exp\,\frac{\left(  d_{1}+d_{2}+d_{3}+d_{4}\right)  ^{4}}{24}\left(
-\left[  X,\left[  X,\left[  X,Y\right]  \right]  \right]  -3\left[  X,\left[
Y,\left[  X,Y\right]  \right]  \right]  -3\left[  Y,\left[  Y,\left[
X,Y\right]  \right]  \right]  \right)
\end{align*}

\end{theorem}

\begin{proof}
We have
\begin{align*}
&  \exp\,\left(  d_{1}+d_{2}+d_{3}+d_{4}\right)  \left(  X+Y\right) \\
&  =\exp\,\left(  d_{1}+d_{2}+d_{3}\right)  \left(  X+Y\right)  +d_{4}\left(
X+Y\right) \\
&  =\exp\,\left(  d_{1}+d_{2}+d_{3}\right)  \left(  X+Y\right)  .\\
&  \left\{  \delta^{\mathrm{left}}\left(  \exp\right)  \left(  \left(
d_{1}+d_{2}+d_{3}\right)  \left(  X+Y\right)  \right)  \left(  \left(
X+Y\right)  \right)  \right\}  _{d_{4}}\\
&  \left)  \text{left logarithmic derivation}\right( \\
&  =\exp\,\left(  d_{1}+d_{2}+d_{3}\right)  \left(  X+Y\right)  .\left(
X+Y\right)  _{d_{4}}\\
&  \left)  \text{By Theorem \ref{t5.2}}\right( \\
&  =\exp\,\left(  d_{1}+d_{2}+d_{3}\right)  \left(  X+Y\right)  .\exp
\,d_{4}\left(  X+Y\right) \\
&  \left)  \text{By Proposition \ref{t4.5}}\right( \\
&  =\exp\,\left(  d_{1}+d_{2}+d_{3}\right)  X.\exp\,\left(  d_{1}+d_{2}%
+d_{3}\right)  Y.\exp\,-\left(  d_{1}d_{2}+d_{1}d_{3}+d_{2}d_{3}\right)
\left[  X,Y\right]  .\\
&  \exp\,d_{1}d_{2}d_{3}\left[  X+2Y,\left[  X,Y\right]  \right]  .\exp
\,d_{4}X.\exp\,d_{4}Y\\
&  \left)  \text{By Theorems \ref{t6.1} and \ref{t6.3}}\right( \\
&  =\exp\,\left(  d_{1}+d_{2}+d_{3}+d_{4}\right)  X.\exp\,d_{1}Y.\exp
\,d_{1}d_{4}\left[  Y,X\right]  .\exp\,d_{2}Y.\exp\,d_{2}d_{4}\left[
Y,X\right]  .\\
&  \exp\,d_{3}Y.\exp\,d_{3}d_{4}\left[  Y,X\right]  .\exp\,-\left(  d_{1}%
d_{2}+d_{1}d_{3}+d_{2}d_{3}\right)  \left[  X,Y\right]  .\\
&  \exp\,-\left(  d_{1}d_{2}+d_{1}d_{3}+d_{2}d_{3}\right)  d_{4}\left[
\left[  X,Y\right]  ,X\right]  .\exp\,d_{1}d_{2}d_{3}\left[  X+2Y,\left[
X,Y\right]  \right]  .\\
&  \exp\,d_{1}d_{2}d_{3}d_{4}\left[  \left[  X+2Y,\left[  X,Y\right]  \right]
,X\right]  .\exp\,d_{4}Y\\
&  \left)
\begin{array}
[c]{c}%
\text{By moving }\exp\,d_{4}X\text{\ left towards }\exp\,\left(  d_{1}%
+d_{2}+d_{3}\right)  X\text{ }\\
\text{via Propositions \ref{t5.3} and \ref{t5.4}}%
\end{array}
\right( \\
&  =\exp\,\left(  d_{1}+d_{2}+d_{3}+d_{4}\right)  X.\exp\,\left(  d_{1}%
+d_{2}\right)  Y.\exp\,d_{1}d_{4}\left[  Y,X\right]  .\exp\,d_{1}d_{2}%
d_{4}\left[  \left[  Y,X\right]  ,Y\right]  .\\
&  \exp\,d_{2}d_{4}\left[  Y,X\right]  .\exp\,d_{3}Y.\exp\,d_{3}d_{4}\left[
Y,X\right]  .\exp\,-\left(  d_{1}d_{2}+d_{1}d_{3}+d_{2}d_{3}\right)  \left[
X,Y\right]  .\\
&  \exp\,-\left(  d_{1}d_{2}+d_{1}d_{3}+d_{2}d_{3}\right)  d_{4}\left[
\left[  X,Y\right]  ,X\right]  .\exp\,d_{1}d_{2}d_{3}\left[  X+2Y,\left[
X,Y\right]  \right]  .\\
&  \exp\,d_{1}d_{2}d_{3}d_{4}\left[  \left[  X+2Y,\left[  X,Y\right]  \right]
,X\right]  .\exp\,d_{4}Y\\
&  \left)  \text{By interchanging }\exp\,d_{1}d_{4}\left[  Y,X\right]
\text{\ and }\exp\,d_{2}Y\text{ via Proposition \ref{t5.4}}\right( \\
&  =\exp\,\left(  d_{1}+d_{2}+d_{3}+d_{4}\right)  X.\exp\,\left(  d_{1}%
+d_{2}+d_{3}\right)  Y.\exp\,d_{1}d_{4}\left[  Y,X\right]  .\exp\,d_{1}%
d_{3}d_{4}\left[  \left[  Y,X\right]  ,Y\right]  .\\
&  \exp\,d_{1}d_{2}d_{4}\left[  \left[  Y,X\right]  ,Y\right]  .\exp
\,d_{1}d_{2}d_{3}d_{4}\left[  \left[  \left[  Y,X\right]  ,Y\right]
,Y\right]  .\exp\,d_{2}d_{4}\left[  Y,X\right]  .\\
&  \exp\,d_{2}d_{3}d_{4}\left[  \left[  Y,X\right]  ,Y\right]  .\exp
\,d_{3}d_{4}\left[  Y,X\right]  .\exp\,-\left(  d_{1}d_{2}+d_{1}d_{3}%
+d_{2}d_{3}\right)  \left[  X,Y\right]  .\\
&  \exp\,-\left(  d_{1}d_{2}+d_{1}d_{3}+d_{2}d_{3}\right)  d_{4}\left[
\left[  X,Y\right]  ,X\right]  .\exp\,d_{1}d_{2}d_{3}\left[  X+2Y,\left[
X,Y\right]  \right]  .\\
&  \exp\,d_{1}d_{2}d_{3}d_{4}\left[  \left[  X+2Y,\left[  X,Y\right]  \right]
,X\right]  .\exp\,d_{4}Y\\
&  \left)
\begin{array}
[c]{c}%
\text{By moving }\exp\,d_{3}Y\text{\ left towards }\exp\,\left(  d_{1}%
+d_{2}\right)  Y\text{ }\\
\text{via Propositions \ref{t5.3} and \ref{t5.4}}%
\end{array}
\right(
\end{align*}
We keep on:
\begin{align*}
&  =\exp\,\left(  d_{1}+d_{2}+d_{3}+d_{4}\right)  X.\exp\,\left(  d_{1}%
+d_{2}+d_{3}+d_{4}\right)  Y.\exp\,d_{1}d_{4}\left[  Y,X\right]  .\\
&  \exp\,d_{1}d_{3}d_{4}\left[  \left[  Y,X\right]  ,Y\right]  .\exp
\,d_{1}d_{2}d_{4}\left[  \left[  Y,X\right]  ,Y\right]  .\exp\,d_{1}d_{2}%
d_{3}d_{4}\left[  \left[  \left[  Y,X\right]  ,Y\right]  ,Y\right]  .\\
&  \exp\,d_{2}d_{4}\left[  Y,X\right]  .\exp\,d_{2}d_{3}d_{4}\left[  \left[
Y,X\right]  ,Y\right]  .\exp\,d_{3}d_{4}\left[  Y,X\right]  .\\
&  \exp\,-\left(  d_{1}d_{2}+d_{1}d_{3}+d_{2}d_{3}\right)  \left[  X,Y\right]
.\exp\,-\left(  d_{1}d_{2}+d_{1}d_{3}+d_{2}d_{3}\right)  d_{4}\left[  \left[
X,Y\right]  ,Y\right]  .\\
&  \exp\,-\left(  d_{1}d_{2}+d_{1}d_{3}+d_{2}d_{3}\right)  d_{4}\left[
\left[  X,Y\right]  ,X\right]  .\exp\,d_{1}d_{2}d_{3}\left[  X+2Y,\left[
X,Y\right]  \right]  .\\
&  \exp\,d_{1}d_{2}d_{3}d_{4}\left[  \left[  X+2Y,\left[  X,Y\right]  \right]
,Y\right]  .\exp\,d_{1}d_{2}d_{3}d_{4}\left[  \left[  X+2Y,\left[  X,Y\right]
\right]  ,X\right] \\
&  \left)
\begin{array}
[c]{c}%
\text{By moving }\exp\,d_{4}Y\text{\ left towards }\exp\,\left(  d_{1}%
+d_{2}+d_{3}\right)  Y\text{ }\\
\text{via Propositions \ref{t5.3} and \ref{t5.4}}%
\end{array}
\right(  \text{ }\\
&  =\exp\,\left(  d_{1}+d_{2}+d_{3}+d_{4}\right)  X.\exp\,\left(  d_{1}%
+d_{2}+d_{3}+d_{4}\right)  Y.\\
&  \exp\,-\left(  d_{1}d_{2}+d_{1}d_{3}+d_{1}d_{4}+d_{2}d_{3}+d_{2}d_{4}%
+d_{3}d_{4}\right)  \left[  X,Y\right]  .\\
&  \exp\,\left(  d_{1}d_{2}d_{3}+d_{1}d_{2}d_{4}+d_{1}d_{3}d_{4}+d_{2}%
d_{3}d_{4}\right)  \left[  X+2Y,\left[  X,Y\right]  \right]  .\\
&  \exp\,d_{1}d_{2}d_{3}d_{4}\left(  -\left[  X,\left[  X,\left[  X,Y\right]
\right]  \right]  -3\left[  X,\left[  Y,\left[  X,Y\right]  \right]  \right]
-3\left[  Y,\left[  Y,\left[  X,Y\right]  \right]  \right]  \right)
\end{align*}
so that we have got to the desired formula.
\end{proof}

We could keep on, but the complexity of computation increases rapidly.

\section{\label{s7}The First Approach to the Baker-Campbell-Hausdorff Formula}

The following result is no other than Theorem \ref{t6.1} itself.

\begin{theorem}
\label{t7.1}Given $X,Y\in\mathfrak{g}$\ and $d_{1}\in D$, we have
\begin{align*}
&  \exp\,d_{1}X.\exp\,d_{1}Y\\
&  =\exp\,d_{1}\left(  X+Y\right)
\end{align*}

\end{theorem}

\begin{corollary}
Given $X_{1},...,X_{n}\in\mathfrak{g}$\ and $d_{1}\in D$, we have
\begin{align*}
&  \exp\,d_{1}X_{1}.\exp\,d_{1}X_{2}....\exp\,d_{1}X_{n}\\
&  =\exp\,d_{1}\left(  X_{1}+X_{2}+...+X_{n}\right)
\end{align*}

\begin{proof}
By simple induction on $n$.
\end{proof}
\end{corollary}

\begin{theorem}
\label{t7.2}Given $X,Y\in\mathfrak{g}$\ and $d_{1},d_{2}\in D$, we have
\begin{align*}
&  \exp\,\left(  d_{1}+d_{2}\right)  X.\exp\,\left(  d_{1}+d_{2}\right)  Y\\
&  =\exp\,\left(  d_{1}+d_{2}\right)  \left(  X+Y\right)  +d_{1}d_{2}\left[
X,Y\right] \\
&  =\exp\,\left(  d_{1}+d_{2}\right)  \left(  X+Y\right)  +\frac{\left(
d_{1}+d_{2}\right)  ^{2}}{2}\left[  X,Y\right]
\end{align*}

\end{theorem}

\begin{proof}
We have
\begin{align*}
&  \exp\,\left(  d_{1}+d_{2}\right)  \left(  X+Y\right) \\
&  =\exp\,d_{1}\left(  X+Y\right)  +d_{2}\left(  X+Y\right) \\
&  =\exp\,d_{1}\left(  X+Y\right)  .\left\{  \delta^{\mathrm{left}}\left(
\exp\right)  \left(  d_{1}\left(  X+Y\right)  \right)  \left(  X+Y\right)
\right\}  _{d_{2}}\\
&  \left)  \text{left logarithmic derivation}\right( \\
&  =\exp\,d_{1}\left(  X+Y\right)  .\left(  X+Y\right)  _{d_{2}}\\
&  \left)  \text{By Theorem \ref{t5.2}}\right( \\
&  =\exp\,d_{1}\left(  X+Y\right)  .\exp\,d_{2}\left(  X+Y\right) \\
&  \left)  \text{By Proposition \ref{t4.5}}\right( \\
&  =\exp\,d_{1}X.\exp\,d_{1}Y.\exp\,d_{2}X.\exp\,d_{2}Y\\
&  \left)  \text{By Theorem \ref{t6.1}}\right( \\
&  =\exp\,d_{1}X.\exp\,d_{2}X.\exp\,d_{1}Y.\exp\,d_{1}d_{2}\left[  Y,X\right]
.\exp\,d_{2}Y\\
&  \left)  \text{By Proposition \ref{t5.4}}\right( \\
&  =\exp\,d_{1}X.\exp\,d_{2}X.\exp\,d_{1}Y.\exp\,d_{2}Y.\exp\,d_{1}d_{2}
\left[  Y,X\right] \\
&  \left)  \text{By Proposition \ref{t5.3}}\right( \\
&  =\exp\,\left(  d_{1}+d_{2}\right)  X.\exp\,\left(  d_{1}+d_{2}\right)
Y.\exp\,d_{1}d_{2}\left[  Y,X\right] \\
&  \left)  \text{By Proposition \ref{t5.3}}\right(
\end{align*}
so that we get the desired formula by multiplying
\[
\exp\,d_{1}d_{2}\left[  X,Y\right]
\]
from the right and making use of Proposition \ref{t5.3}.
\end{proof}

\begin{corollary}
\label{t7.2.1}(cf. Theorem 2.12.4 of \cite{va}). Given $X_{1},...,X_{n}%
\in\mathfrak{g}$\ and $d_{1},d_{2}\in D$, we have
\begin{align*}
&  \exp\,\left(  d_{1}+d_{2}\right)  X_{1}.\exp\,\left(  d_{1}+d_{2}\right)
X_{2}....\exp\,\left(  d_{1}+d_{2}\right)  X_{n}\\
&  =\exp\,\left(  d_{1}+d_{2}\right)  \left(  X_{1}+...+X_{n}\right)
+d_{1}d_{2}\sum_{1\leq i<j\leq n}\left[  X_{i},X_{j}\right] \\
&  =\exp\,\left(  d_{1}+d_{2}\right)  \left(  X_{1}+...+X_{n}\right)
+\frac{\left(  d_{1}+d_{2}\right)  ^{2}}{2}\sum_{1\leq i<j\leq n}\left[
X_{i},X_{j}\right]
\end{align*}

\end{corollary}

\begin{proof}
Here we deal only with the case of $n=3$, leaving the general treatment by
induction on $n$\ to the reader. We note in passing that the case of $n=2$\ is
no other than Theorem \ref{t7.2} itself. We have
\begin{align*}
&  \exp\,\left(  d_{1}+d_{2}\right)  X_{1}.\exp\,\left(  d_{1}+d_{2}\right)
X_{2}.\exp\,\left(  d_{1}+d_{2}\right)  X_{3}\\
&  =\exp\,\left(  d_{1}+d_{2}\right)  \left(  X_{1}+X_{2}\right)
+\frac{\left(  d_{1}+d_{2}\right)  ^{2}}{2}\left[  X_{1},X_{2}\right]
.\exp\,\left(  d_{1}+d_{2}\right)  X_{3}\\
&  \left)  \text{By Theorem \ref{t7.2}}\right( \\
&  =\exp\,\left(  d_{1}+d_{2}\right)  \left\{  \left(  X_{1}+X_{2}\right)
+\frac{d_{1}+d_{2}}{2}\left[  X_{1},X_{2}\right]  \right\}  .\exp\,\left(
d_{1}+d_{2}\right)  X_{3}\\
&  =\exp\,\left(  d_{1}+d_{2}\right)  \left\{  \left(  X_{1}+X_{2}%
+X_{3}\right)  +\frac{d_{1}+d_{2}}{2}\left[  X_{1},X_{2}\right]  \right\}  +\\
&  d_{1}d_{2}\left[  \left(  X_{1}+X_{2}\right)  +\frac{d_{1}+d_{2}}{2}\left[
X_{1},X_{2}\right]  ,X_{3}\right] \\
&  \left)  \text{By Theorem \ref{t7.2}}\right( \\
&  =\exp\,\left(  d_{1}+d_{2}\right)  \left(  X_{1}+X_{2}+X_{3}\right)
+d_{1}d_{2}\left(  \left[  X_{1},X_{2}\right]  +\left[  X_{1},X_{3}\right]
+\left[  X_{2},X_{3}\right]  \right)
\end{align*}
so that we are done.
\end{proof}

\begin{theorem}
\label{t7.3}Given $X,Y\in\mathfrak{g}$\ and $d_{1},d_{2},d_{3}\in D$, we have
\begin{align*}
&  \exp\,\left(  d_{1}+d_{2}+d_{3}\right)  X.\exp\,\left(  d_{1}+d_{2}%
+d_{3}\right)  Y\\
&  =\exp\,\left(  d_{1}+d_{2}+d_{3}\right)  \left(  X+Y\right)  +\left(
d_{1}d_{2}+d_{1}d_{3}+d_{2}d_{3}\right)  \left[  X,Y\right]  +\\
&  \frac{1}{2}d_{1}d_{2}d_{3}\left[  X-Y,\left[  X,Y\right]  \right] \\
&  =\exp\,\left(  d_{1}+d_{2}+d_{3}\right)  X+\left(  d_{1}+d_{2}%
+d_{3}\right)  Y+\frac{\left(  d_{1}+d_{2}+d_{3}\right)  ^{2}}{2}\left[
X,Y\right]  +\\
&  \frac{\left(  d_{1}+d_{2}+d_{3}\right)  ^{3}}{12}\left[  X-Y,\left[
X,Y\right]  \right]
\end{align*}

\end{theorem}

\begin{proof}
We have
\begin{align*}
&  \exp\,\left(  d_{1}+d_{2}+d_{3}\right)  \left(  X+Y\right)  +\left(
d_{1}d_{2}+d_{1}d_{3}+d_{2}d_{3}\right)  \left[  X,Y\right] \\
&  =\exp\,\left\{  \left(  d_{1}+d_{2}\right)  \left(  X+Y\right)  +d_{1}d_{2}
\left[  X,Y\right]  \right\}  +d_{3}\left\{  \left(  X+Y\right)  +\left(
d_{1}+d_{2}\right)  \left[  X,Y\right]  \right\} \\
&  =\exp\,\left(  \left(  d_{1}+d_{2}\right)  \left(  X+Y\right)  +d_{1}d_{2}
\left[  X,Y\right]  \right)  .\\
&  \left\{  \delta^{\mathrm{left}}\left(  \exp\right)  \left(  \left(
d_{1}+d_{2}\right)  \left(  X+Y\right)  +d_{1}d_{2}\left[  X,Y\right]
\right)  \left(  \left(  X+Y\right)  +\left(  d_{1}+d_{2}\right)  \left[
X,Y\right]  \right)  \right\}  _{d_{3}}\\
&  \left)  \text{left logarithmic derivation}\right( \\
&  =\exp\,\left(  \left(  d_{1}+d_{2}\right)  \left(  X+Y\right)  +d_{1}d_{2}
\left[  X,Y\right]  \right)  .\\
&  \left\{  \left(  X+Y\right)  +\left(  d_{1}+d_{2}\right)  \left[
X,Y\right]  -\frac{1}{2}d_{1}d_{2}\left[  X+Y,\left[  X,Y\right]  \right]
\right\}  _{d_{3}}\\
&  \left)  \text{By Theorem \ref{t5.2}}\right( \\
&  =\exp\,\left(  d_{1}+d_{2}\right)  X.\exp\,\left(  d_{1}+d_{2}\right)  Y.\\
&  \left\{  \left(  X+Y\right)  +\left(  d_{1}+d_{2}\right)  \left[
X,Y\right]  -\frac{1}{2}d_{1}d_{2}\left[  X+Y,\left[  X,Y\right]  \right]
\right\}  _{d_{3}}\\
&  \text{[By Theorem \ref{t7.2}]}\\
&  =\exp\,\left(  d_{1}+d_{2}\right)  X.\exp\,\left(  d_{1}+d_{2}\right)
Y.\left(  X+Y\right)  _{d_{3}}.\left(  \left(  d_{1}+d_{2}\right)  \left[
X,Y\right]  \right)  _{d_{3}}.\\
&  \left(  -\frac{1}{2}d_{1}d_{2}\left[  X+Y,\left[  X,Y\right]  \right]
\right)  _{d_{3}}\\
&  \text{[By Proposition \ref{t2.2}]}\\
&  =\exp\,\left(  d_{1}+d_{2}\right)  X.\exp\,\left(  d_{1}+d_{2}\right)
Y.\exp\,d_{3}\left(  X+Y\right)  .\exp\,\left(  d_{1}+d_{2}\right)
d_{3}\left[  X,Y\right]  .\\
&  \exp\,-\frac{1}{2}d_{1}d_{2}d_{3}\left[  X+Y,\left[  X,Y\right]  \right] \\
&  \text{[By Proposition \ref{t4.5}]}\\
&  =\exp\,\left(  d_{1}+d_{2}\right)  X.\exp\,d_{1}Y.\exp\,d_{2}Y.\exp
\,d_{3}X.\exp\,d_{3}Y.\exp\,\left(  d_{1}+d_{2}\right)  d_{3}\left[
X,Y\right]  .\\
&  \exp\,-\frac{1}{2}d_{1}d_{2}d_{3}\left[  X+Y,\left[  X,Y\right]  \right] \\
&  \left)  \text{By Proposition \ref{t5.3}}\right( \\
&  =\exp\,\left(  d_{1}+d_{2}\right)  X.\exp\,d_{1}Y.\exp\,d_{3}X.\exp
\,d_{2}Y.\exp\,d_{2}d_{3}\left[  Y,X\right]  .\exp\,d_{3}Y.\\
&  \exp\,\left(  d_{1}+d_{2}\right)  d_{3}\left[  X,Y\right]  .\exp\,-\frac
{1}{2}d_{1}d_{2}d_{3}\left[  X+Y,\left[  X,Y\right]  \right] \\
&  \left)  \text{By Proposition \ref{t5.4}}\right( \\
&  =\exp\,\left(  d_{1}+d_{2}\right)  X.\exp\,d_{3}X.\exp\,d_{1}Y.\exp
\,d_{1}d_{3}\left[  Y,X\right]  .\exp\,d_{2}Y.\exp\,d_{2}d_{3}\left[
Y,X\right]  .\exp\,d_{3}Y.\\
&  \exp\,\left(  d_{1}+d_{2}\right)  d_{3}\left[  X,Y\right]  .\exp\,-\frac
{1}{2}d_{1}d_{2}d_{3}\left[  X+Y,\left[  X,Y\right]  \right] \\
&  \left)  \text{By Proposition \ref{t5.4}}\right(
\end{align*}
We keep on.
\begin{align*}
&  =\exp\,\left(  d_{1}+d_{2}\right)  X.\exp\,d_{3}X.\exp\,d_{1}Y.\exp
\,d_{2}Y.\exp\,d_{1}d_{3}\left[  Y,X\right]  .\exp\,d_{1}d_{2}d_{3}\left[
\left[  Y,X\right]  ,Y\right]  .\\
&  \exp\,d_{2}d_{3}\left[  Y,X\right]  .\exp\,d_{3}Y.\exp\,\left(  d_{1}%
+d_{2}\right)  d_{3}\left[  X,Y\right]  .\exp\,-\frac{1}{2}d_{1}d_{2}%
d_{3}\left[  X+Y,\left[  X,Y\right]  \right] \\
&  \left)  \text{By Proposition \ref{t5.4}}\right( \\
&  =\exp\,\left(  d_{1}+d_{2}+d_{3}\right)  X.\exp\,\left(  d_{1}+d_{2}%
+d_{3}\right)  Y.\exp\,\frac{1}{2}d_{1}d_{2}d_{3}\left[  Y-X,\left[
X,Y\right]  \right] \\
&  \left)  \text{By repeated use of Proposition \ref{t5.3}}\right(
\end{align*}

so that we get the desired formula by multiplying
\[
\exp\,\frac{1}{2}d_{1}d_{2}d_{3}\left[  X-Y,\left[  X,Y\right]  \right]
\]
from the right and making use of Proposition \ref{t5.3}.
\end{proof}

\begin{theorem}
\label{t7.4}Given $X,Y\in\mathfrak{g}$\ and $d_{1},d_{2},d_{3},d_{4}\in D$, we
have
\begin{align*}
&  \exp\,\left(  d_{1}+d_{2}+d_{3}+d_{4}\right)  X.\exp\,\left(  d_{1}%
+d_{2}+d_{3}+d_{4}\right)  Y\\
&  =\exp\,\left(  d_{1}+d_{2}+d_{3}+d_{4}\right)  X+\left(  d_{1}+d_{2}%
+d_{3}+d_{4}\right)  Y+\\
&  \left(  d_{1}d_{2}+d_{1}d_{3}+d_{1}d_{4}+d_{2}d_{3}+d_{2}d_{4}+d_{3}%
d_{4}\right)  \left[  X,Y\right]  +\\
&  \frac{1}{2}\left(  d_{1}d_{2}d_{3}+d_{1}d_{2}d_{4}+d_{1}d_{3}d_{4}%
+d_{2}d_{3}d_{4}\right)  \left[  X-Y,\left[  X,Y\right]  \right]  -\\
&  d_{1}d_{2}d_{3}d_{4}\left(  \frac{1}{2}\left[  X,\left[  X,\left[
X,Y\right]  \right]  \right]  +\frac{1}{2}\left[  Y,\left[  Y,\left[
X,Y\right]  \right]  \right]  +2\left[  X,\left[  Y,\left[  X,Y\right]
\right]  \right]  \right) \\
&  =\exp\,\left(  d_{1}+d_{2}+d_{3}+d_{4}\right)  X+\left(  d_{1}+d_{2}%
+d_{3}+d_{4}\right)  Y+\\
&  \frac{\left(  d_{1}+d_{2}+d_{3}+d_{4}\right)  ^{2}}{2}\left[  X,Y\right]
+\\
&  \frac{\left(  d_{1}+d_{2}+d_{3}+d_{4}\right)  ^{3}}{12}\left[  X-Y,\left[
X,Y\right]  \right]  -\\
&  \frac{\left(  d_{1}+d_{2}+d_{3}+d_{4}\right)  ^{4}}{24}\left(  \frac{1}%
{2}\left[  X,\left[  X,\left[  X,Y\right]  \right]  \right]  +\frac{1}%
{2}\left[  Y,\left[  Y,\left[  X,Y\right]  \right]  \right]  +2\left[
X,\left[  Y,\left[  X,Y\right]  \right]  \right]  \right)
\end{align*}

\end{theorem}

\begin{proof}
We have
\begin{align*}
&  \exp\,\left(  d_{1}+d_{2}+d_{3}+d_{4}\right)  X+\left(  d_{1}+d_{2}%
+d_{3}+d_{4}\right)  Y+\\
&  \left(  d_{1}d_{2}+d_{1}d_{3}+d_{1}d_{4}+d_{2}d_{3}+d_{2}d_{4}+d_{3}%
d_{4}\right)  \left[  X,Y\right]  +\\
&  \frac{1}{2}\left(  d_{1}d_{2}d_{3}+d_{1}d_{2}d_{4}+d_{1}d_{3}d_{4}%
+d_{2}d_{3}d_{4}\right)  \left[  X-Y,\left[  X,Y\right]  \right] \\
&  =\exp\,\left\{
\begin{array}
[c]{c}%
\left(  d_{1}+d_{2}+d_{3}\right)  \left(  X+Y\right)  +\left(  d_{1}%
d_{2}+d_{1}d_{3}+d_{2}d_{3}\right)  \left[  X,Y\right]  +\\
\frac{1}{2}d_{1}d_{2}d_{3}\left[  X-Y,\left[  X,Y\right]  \right]
\end{array}
\right\}  +\\
&  \left\{  d_{4}\left(  X+Y\right)  +d_{4}\left(  d_{1}+d_{2}+d_{3}\right)
\left[  X,Y\right]  +\frac{1}{2}d_{4}\left(  d_{1}d_{2}+d_{1}d_{3}+d_{2}%
d_{3}\right)  \left[  X-Y,\left[  X,Y\right]  \right]  \right\} \\
&  =\exp\,\left\{
\begin{array}
[c]{c}%
\left(  d_{1}+d_{2}+d_{3}\right)  X+\left(  d_{1}+d_{2}+d_{3}\right)
Y+\left(  d_{1}d_{2}+d_{1}d_{3}+d_{2}d_{3}\right)  \left[  X,Y\right]  +\\
\frac{1}{2}d_{1}d_{2}d_{3}\left[  X-Y,\left[  X,Y\right]  \right]
\end{array}
\right\}  .\\
&  \left\{
\begin{array}
[c]{c}%
\delta^{\mathrm{left}}\left(  \exp\right)  \left(
\begin{array}
[c]{c}%
\left(  d_{1}+d_{2}+d_{3}\right)  \left(  X+Y\right)  +\left(  d_{1}%
d_{2}+d_{1}d_{3}+d_{2}d_{3}\right)  \left[  X,Y\right]  +\\
\frac{1}{2}d_{1}d_{2}d_{3}\left[  X-Y,\left[  X,Y\right]  \right]
\end{array}
\right) \\
\left(
\begin{array}
[c]{c}%
\left(  X+Y\right)  +\left(  d_{1}+d_{2}+d_{3}\right)  \left[  X,Y\right]  +\\
\frac{1}{2}\left(  d_{1}d_{2}+d_{1}d_{3}+d_{2}d_{3}\right)  \left[
X-Y,\left[  X,Y\right]  \right]
\end{array}
\right)
\end{array}
\right\}  _{d_{4}}\\
&  \left)  \text{left logarithmic derivation}\right( \\
&  =\exp\,\left\{
\begin{array}
[c]{c}%
\left(  d_{1}+d_{2}+d_{3}\right)  \left(  X+Y\right)  +\left(  d_{1}%
d_{2}+d_{1}d_{3}+d_{2}d_{3}\right)  \left[  X,Y\right]  +\\
\frac{1}{2}d_{1}d_{2}d_{3}\left[  X-Y,\left[  X,Y\right]  \right]
\end{array}
\right\}  .\\
&  \left\{
\begin{array}
[c]{c}%
\left(  X+Y\right)  +\left(  d_{1}+d_{2}+d_{3}\right)  \left[  X,Y\right]
+\frac{1}{2}\left(  d_{1}d_{2}+d_{1}d_{3}+d_{2}d_{3}\right)  \left[
X-Y,\left[  X,Y\right]  \right]  -\\
\frac{1}{2}\left(  d_{1}+d_{2}+d_{3}\right)  ^{2}\left[  X+Y,\left[
X,Y\right]  \right]  -\frac{1}{8}\left(  d_{1}+d_{2}+d_{3}\right)  ^{3}\left[
X+Y,\left[  X-Y,\left[  X,Y\right]  \right]  \right]  +\\
\frac{1}{2}\left(  d_{1}d_{2}+d_{1}d_{3}+d_{2}d_{3}\right)  \left[
X+Y,\left[  X,Y\right]  \right]  +\frac{1}{4}d_{1}d_{2}d_{3}\left[
X+Y,\left[  X-Y,\left[  X,Y\right]  \right]  \right]  +\\
\frac{1}{6}\left(  d_{1}+d_{2}+d_{3}\right)  ^{3}\left[  X+Y,\left[
X+Y,\left[  X,Y\right]  \right]  \right]
\end{array}
\right\}  _{d_{4}}\\
&  \left)  \text{By Theorem \ref{t5.2}}\right( \\
&  =\exp\,\left\{
\begin{array}
[c]{c}%
\left(  d_{1}+d_{2}+d_{3}\right)  \left(  X+Y\right)  +\left(  d_{1}%
d_{2}+d_{1}d_{3}+d_{2}d_{3}\right)  \left[  X,Y\right]  +\\
\frac{1}{2}d_{1}d_{2}d_{3}\left[  X-Y,\left[  X,Y\right]  \right]
\end{array}
\right\}  .\\
&  \left\{
\begin{array}
[c]{c}%
\left(  X+Y\right)  +\left(  d_{1}+d_{2}+d_{3}\right)  \left[  X,Y\right]
-\left(  d_{1}d_{2}+d_{1}d_{3}+d_{2}d_{3}\right)  \left[  Y,\left[
X,Y\right]  \right]  +\\
\frac{1}{2}d_{1}d_{2}d_{3}\left[  X+Y,\left[  X,\left[  X,Y\right]  \right]
\right]  +\frac{3}{2}d_{1}d_{2}d_{3}\left[  X+Y,\left[  Y,\left[  X,Y\right]
\right]  \right]
\end{array}
\right\}  _{d_{4}}\\
&  =\exp\,\left(  d_{1}+d_{2}+d_{3}\right)  X.\exp\,\left(  d_{1}+d_{2}%
+d_{3}\right)  Y.\\
&  \left\{
\begin{array}
[c]{c}%
\left(  X+Y\right)  +\left(  d_{1}+d_{2}+d_{3}\right)  \left[  X,Y\right]
-\left(  d_{1}d_{2}+d_{1}d_{3}+d_{2}d_{3}\right)  \left[  Y,\left[
X,Y\right]  \right]  +\\
\frac{1}{2}d_{1}d_{2}d_{3}\left[  X+Y,\left[  X,\left[  X,Y\right]  \right]
\right]  +\frac{3}{2}d_{1}d_{2}d_{3}\left[  X+Y,\left[  Y,\left[  X,Y\right]
\right]  \right]
\end{array}
\right\}  _{d_{4}}\\
&  \left)  \text{By Theorem \ref{t7.3}}\right( \\
&  =\exp\,\left(  d_{1}+d_{2}+d_{3}\right)  X.\exp\,\left(  d_{1}+d_{2}%
+d_{3}\right)  Y.\left(  X+Y\right)  _{d_{4}}.\left(  \left(  d_{1}%
+d_{2}+d_{3}\right)  \left[  X,Y\right]  \right)  _{d_{4}}.\\
&  \left(  -\left(  d_{1}d_{2}+d_{1}d_{3}+d_{2}d_{3}\right)  \left[  Y,\left[
X,Y\right]  \right]  \right)  _{d_{4}}.\\
&  \left(  \frac{1}{2}d_{1}d_{2}d_{3}\left[  X+Y,\left[  X,\left[  X,Y\right]
\right]  \right]  +\frac{3}{2}d_{1}d_{2}d_{3}\left[  X+Y,\left[  Y,\left[
X,Y\right]  \right]  \right]  \right)  _{d_{4}}\\
&  \left)  \text{By Proposition \ref{t2.2}}\right( \\
&  =\exp\,\left(  d_{1}+d_{2}+d_{3}\right)  X.\exp\,\left(  d_{1}+d_{2}%
+d_{3}\right)  Y.\exp\,d_{4}\left(  X+Y\right)  .\exp\,d_{4}\left(
d_{1}+d_{2}+d_{3}\right)  \left[  X,Y\right]  .\\
&  \exp\,-d_{4}\left(  d_{1}d_{2}+d_{1}d_{3}+d_{2}d_{3}\right)  \left[
Y,\left[  X,Y\right]  \right]  .\\
&  \exp\,d_{4}\left(  \frac{1}{2}d_{1}d_{2}d_{3}\left[  X+Y,\left[  X,\left[
X,Y\right]  \right]  \right]  +\frac{3}{2}d_{1}d_{2}d_{3}\left[  X+Y,\left[
Y,\left[  X,Y\right]  \right]  \right]  \right) \\
&  \left)  \text{By Proposition \ref{t4.5}}\right( \\
&  =\exp\,\left(  d_{1}+d_{2}+d_{3}\right)  X.\exp\,d_{1}Y.\exp\,d_{2}%
Y.\exp\,d_{3}Y.\exp\,d_{4}X.\exp\,d_{4}Y.\\
&  \exp\,d_{4}\left(  d_{1}+d_{2}+d_{3}\right)  \left[  X,Y\right]
.\exp\,-d_{4}\left(  d_{1}d_{2}+d_{1}d_{3}+d_{2}d_{3}\right)  \left[
Y,\left[  X,Y\right]  \right]  .\\
&  \exp\,d_{4}\left(  \frac{1}{2}d_{1}d_{2}d_{3}\left[  X+Y,\left[  X,\left[
X,Y\right]  \right]  \right]  +\frac{3}{2}d_{1}d_{2}d_{3}\left[  X+Y,\left[
Y,\left[  X,Y\right]  \right]  \right]  \right) \\
&  \left)  \text{By Proposition \ref{t5.3}}\right(
\end{align*}
We keep on
\begin{align*}
&  =\exp\,\left(  d_{1}+d_{2}+d_{3}+d_{4}\right)  X.\exp\,d_{1}Y.\exp
\,d_{1}d_{4}\left[  Y,X\right]  .\exp\,d_{2}Y.\exp\,d_{2}d_{4}\left[
Y,X\right]  .\exp\,d_{3}Y.\\
&  \exp\,d_{3}d_{4}\left[  Y,X\right]  .\exp\,d_{4}Y.\exp\,d_{4}\left(
d_{1}+d_{2}+d_{3}\right)  \left[  X,Y\right]  .\\
&  \exp\,-d_{4}\left(  d_{1}d_{2}+d_{1}d_{3}+d_{2}d_{3}\right)  \left[
Y,\left[  X,Y\right]  \right]  .\\
&  \exp\,d_{4}\left(  \frac{1}{2}d_{1}d_{2}d_{3}\left[  X+Y,\left[  X,\left[
X,Y\right]  \right]  \right]  +\frac{3}{2}d_{1}d_{2}d_{3}\left[  X+Y,\left[
Y,\left[  X,Y\right]  \right]  \right]  \right) \\
&  \left)
\begin{array}
[c]{c}%
\text{By moving }\exp\,d_{4}X\text{\ left towards }\exp\,\left(  d_{1}%
+d_{2}+d_{3}\right)  X\text{\ }\\
\text{via Propositions \ref{t5.3} and \ref{t5.4}}%
\end{array}
\right( \\
&  =\exp\,\left(  d_{1}+d_{2}+d_{3}+d_{4}\right)  X.\exp\,\left(  d_{1}%
+d_{2}\right)  Y.\exp\,d_{1}d_{4}\left[  Y,X\right]  .\exp\,d_{1}d_{2}%
d_{4}\left[  \left[  Y,X\right]  ,Y\right]  .\\
&  \exp\,d_{2}d_{4}\left[  Y,X\right]  .\exp\,d_{3}Y.\exp\,d_{3}d_{4}\left[
Y,X\right]  .\exp\,d_{4}Y.\exp\,d_{4}\left(  d_{1}+d_{2}+d_{3}\right)  \left[
X,Y\right]  .\\
&  \exp\,-d_{4}\left(  d_{1}d_{2}+d_{1}d_{3}+d_{2}d_{3}\right)  \left[
Y,\left[  X,Y\right]  \right]  .\\
&  \exp\,d_{4}\left(  \frac{1}{2}d_{1}d_{2}d_{3}\left[  X+Y,\left[  X,\left[
X,Y\right]  \right]  \right]  +\frac{3}{2}d_{1}d_{2}d_{3}\left[  X+Y,\left[
Y,\left[  X,Y\right]  \right]  \right]  \right) \\
&  \left)
\begin{array}
[c]{c}%
\text{By exchanging }\exp\,d_{1}d_{4}\left[  Y,X\right]  \text{ and }%
\exp\,d_{2}Y\text{\ via Proposition \ref{t5.4} }\\
\text{and using Proposition \ref{t5.3}}%
\end{array}
\right( \\
&  =\exp\,\left(  d_{1}+d_{2}+d_{3}+d_{4}\right)  X.\exp\,\left(  d_{1}%
+d_{2}+d_{3}\right)  Y.\exp\,d_{1}d_{4}\left[  Y,X\right]  .\exp\,d_{1}%
d_{3}d_{4}\left[  \left[  Y,X\right]  ,Y\right]  .\\
&  \exp\,d_{1}d_{2}d_{4}\left[  \left[  Y,X\right]  ,Y\right]  .\exp
\,d_{1}d_{2}d_{3}d_{4}\left[  \left[  \left[  Y,X\right]  ,Y\right]
,Y\right]  .\exp\,d_{2}d_{4}\left[  Y,X\right]  .\\
&  \exp\,d_{2}d_{3}d_{4}\left[  \left[  Y,X\right]  ,Y\right]  .\exp
\,d_{3}d_{4}\left[  Y,X\right]  .\exp\,d_{4}Y.\exp\,d_{4}\left(  d_{1}%
+d_{2}+d_{3}\right)  \left[  X,Y\right]  .\\
&  \exp\,-d_{4}\left(  d_{1}d_{2}+d_{1}d_{3}+d_{2}d_{3}\right)  \left[
Y,\left[  X,Y\right]  \right]  .\\
&  \exp\,d_{4}\left(  \frac{1}{2}d_{1}d_{2}d_{3}\left[  X+Y,\left[  X,\left[
X,Y\right]  \right]  \right]  +\frac{3}{2}d_{1}d_{2}d_{3}\left[  X+Y,\left[
Y,\left[  X,Y\right]  \right]  \right]  \right) \\
&  \left)
\begin{array}
[c]{c}%
\text{By moving }\exp\,d_{3}Y\text{\ left towards }\exp\,\left(  d_{1}%
+d_{2}\right)  Y\text{\ }\\
\text{via Propositions \ref{t5.3} and \ref{t5.4}}%
\end{array}
\right( \\
&  =\exp\,\left(  d_{1}+d_{2}+d_{3}+d_{4}\right)  X.\exp\,\left(  d_{1}%
+d_{2}+d_{3}+d_{4}\right)  Y.\\
&  \exp\,d_{1}d_{2}d_{3}d_{4}\left(  \left[  \left[  \left[  Y,X\right]
,Y\right]  ,Y\right]  +\frac{1}{2}\left[  X+Y,\left[  X,\left[  X,Y\right]
\right]  \right]  +\frac{3}{2}\left[  X+Y,\left[  Y,\left[  X,Y\right]
\right]  \right]  \right) \\
&  \left)  \text{By moving }\exp\,d_{4}Y\text{\ left towards }\exp\,\left(
d_{1}+d_{2}+d_{3}\right)  Y\text{\ via Proposition \ref{t5.3}}\right( \\
&  =\exp\,\left(  d_{1}+d_{2}+d_{3}+d_{4}\right)  X.\exp\,\left(  d_{1}%
+d_{2}+d_{3}+d_{4}\right)  Y.\\
&  \exp\,d_{1}d_{2}d_{3}d_{4}\left(
\begin{array}
[c]{c}%
\frac{1}{2}\left[  X,\left[  X,\left[  X,Y\right]  \right]  \right]  +\frac
{1}{2}\left[  Y,\left[  X,\left[  X,Y\right]  \right]  \right]  +\frac{3}%
{2}\left[  X,\left[  Y,\left[  X,Y\right]  \right]  \right]  +\\
\frac{1}{2}\left[  Y,\left[  Y,\left[  X,Y\right]  \right]  \right]
\end{array}
\right) \\
&  =\exp\,\left(  d_{1}+d_{2}+d_{3}+d_{4}\right)  X.\exp\,\left(  d_{1}%
+d_{2}+d_{3}+d_{4}\right)  Y.\\
&  \exp\,d_{1}d_{2}d_{3}d_{4}\left(  \frac{1}{2}\left[  X,\left[  X,\left[
X,Y\right]  \right]  \right]  +\frac{1}{2}\left[  Y,\left[  Y,\left[
X,Y\right]  \right]  \right]  +2\left[  X,\left[  Y,\left[  X,Y\right]
\right]  \right]  \right) \\
&  \left)  \text{By Lemma \ref{t2.5}}\right(
\end{align*}
so that we get the desired formula by multiplying
\[
\exp\,-d_{1}d_{2}d_{3}d_{4}\left(  \frac{1}{2}\left[  X,\left[  X,\left[
X,Y\right]  \right]  \right]  +\frac{1}{2}\left[  Y,\left[  Y,\left[
X,Y\right]  \right]  \right]  +2\left[  X,\left[  Y,\left[  X,Y\right]
\right]  \right]  \right)
\]
from the right and making use of Proposition \ref{t5.3}.
\end{proof}

We could keep on, but the complexity of computation increases rapidly.

\section{\label{s8}The Second Approach to the Baker-Campbell-Hausdorff
Formula}

\begin{theorem}
\label{t8.1}Given $X,Y\in\mathfrak{g}$\ and $d_{1}\in D$, we have
\begin{align*}
&  \exp\,d_{1}X.\exp\,\,d_{1}Y\\
&  =\exp\,d_{1}\left(  X+Y\right)
\end{align*}

\end{theorem}

\begin{proof}
By Proposition \ref{t5.3}.
\end{proof}

\begin{theorem}
\label{t8.2}Given $X,Y\in\mathfrak{g}$\ and $d_{1},d_{2}\in D$, we have
\begin{align*}
&  \exp\,\left(  d_{1}+d_{2}\right)  X.\exp\,\left(  d_{1}+d_{2}\right)  Y\\
&  =\exp\,\left(  d_{1}+d_{2}\right)  \left(  X+Y\right)  +\frac{1}{2}\left(
d_{1}+d_{2}\right)  ^{2}\left[  X,Y\right]
\end{align*}

\end{theorem}

\begin{proof}
We have
\begin{align*}
&  \exp\,\left(  d_{1}+d_{2}\right)  X.\exp\,\left(  d_{1}+d_{2}\right)  Y\\
&  =\exp\,d_{1}X+d_{2}X.\exp\,d_{1}Y+d_{2}Y\\
&  =\exp\,d_{2}X.\exp\,d_{1}X.\exp\,d_{1}Y.\exp\,d_{2}Y\\
&  \left)  \text{By Proposition \ref{t5.3}}\right( \\
&  =\exp\,d_{2}X.\exp\,d_{1}\left(  X+Y\right)  .\exp\,d_{2}Y\\
&  \left)  \text{By Theorem \ref{t8.1}}\right( \\
&  =\exp\,d_{2}X.\exp\,d_{1}\left(  X+Y\right)  .\exp\,d_{2}\left\{
Y-\frac{1}{2}d_{1}\left[  X,Y\right]  \right\}  .\exp\,\frac{1}{2}d_{1}%
d_{2}\left[  X,Y\right] \\
&  =\exp\,d_{2}X.\exp\,d_{1}\left(  X+Y\right)  +d_{2}Y.\exp\,\frac{1}{2}%
d_{1}d_{2}\left[  X,Y\right] \\
&  \left)  \text{By Theorem \ref{t5.2}\ with }\delta^{\mathrm{left}}\left(
\exp\right)  \left(  d_{1}\left(  X+Y\right)  \right)  \left(  Y\right)
=Y-\frac{1}{2}d_{1}\left[  X,Y\right]  \right( \\
&  =\exp\,-\frac{1}{2}d_{1}d_{2}\left[  Y,X\right]  .\exp\,d_{2}\left\{
X+\frac{1}{2}d_{1}\left[  Y,X\right]  \right\}  .\exp\,d_{1}\left(
X+Y\right)  +d_{2}Y.\exp\,\frac{1}{2}d_{1}d_{2}\left[  X,Y\right] \\
&  =\exp\,-\frac{1}{2}d_{1}d_{2}\left[  Y,X\right]  .\exp\,\left(  d_{1}%
+d_{2}\right)  \left(  X+Y\right)  .\exp\,\frac{1}{2}d_{1}d_{2}\left[
X,Y\right] \\
&  \left)  \text{By Theorem \ref{t5.6}\ with }\delta^{\mathrm{right}}\left(
\exp\right)  \left(  d_{1}\left(  X+Y\right)  +d_{2}X\right)  \left(
X\right)  =X+\frac{1}{2}d_{1}\left[  Y,X\right]  \right( \\
&  =\exp\,\left(  d_{1}+d_{2}\right)  \left(  X+Y\right)  +d_{1}d_{2}\left[
X,Y\right]
\end{align*}

\end{proof}

\begin{theorem}
\label{t8.3}Given $X,Y\in\mathfrak{g}$\ and $d_{1},d_{2},d_{3}\in D$, we have
\begin{align*}
&  \exp\,\left(  d_{1}+d_{2}+d_{3}\right)  X.\exp\,\left(  d_{1}+d_{2}%
+d_{3}\right)  Y\\
&  =\exp\,\left(  d_{1}+d_{2}+d_{3}\right)  \left(  X+Y\right)  +\frac{1}%
{2}\left(  d_{1}+d_{2}+d_{3}\right)  ^{2}\left[  X,Y\right]  +\\
&  \frac{1}{12}\left(  d_{1}+d_{2}+d_{3}\right)  ^{3}\left[  X-Y,\left[
X,Y\right]  \right]
\end{align*}

\end{theorem}

\begin{proof}
We have
\begin{align*}
&  \exp\,\left(  d_{1}+d_{2}+d_{3}\right)  X.\exp\,\left(  d_{1}+d_{2}%
+d_{3}\right)  Y\\
&  =\exp\,\left(  d_{1}+d_{2}\right)  X+d_{3}X.\exp\,\left(  d_{1}%
+d_{2}\right)  Y+d_{3}Y\\
&  =\exp\,d_{3}X.\exp\,\left(  d_{1}+d_{2}\right)  X.\exp\,\left(  d_{1}%
+d_{2}\right)  Y.\exp\,d_{3}Y\\
&  \left)  \text{By Proposition \ref{t5.3}}\right(  \\
&  =\exp\,d_{3}X.\exp\,\left(  d_{1}+d_{2}\right)  \left(  X+Y\right)
+\frac{1}{2}\left(  d_{1}+d_{2}\right)  ^{2}\left[  X,Y\right]  .\exp
\,d_{3}Y\\
&  \left)  \text{By Theorem \ref{t8.2}}\right(  \\
&  =\exp\,d_{3}X.\exp\,\left(  d_{1}+d_{2}\right)  \left(  X+Y\right)
+\frac{1}{2}\left(  d_{1}+d_{2}\right)  ^{2}\left[  X,Y\right]  .\\
&  \exp\,d_{3}\left\{
\begin{array}
[c]{c}%
Y-\frac{1}{2}\left(  \left(  d_{1}+d_{2}\right)  \left[  X,Y\right]  +\frac
{1}{2}\left(  d_{1}+d_{2}\right)  ^{2}\left[  \left[  X,Y\right]  ,Y\right]
\right)  +\\
\frac{1}{6}\left(  d_{1}+d_{2}\right)  ^{2}\left[  X+Y,\left[  X,Y\right]
\right]
\end{array}
\right\}  .\\
&  \exp\,d_{3}\left\{
\begin{array}
[c]{c}%
\frac{1}{2}\left(  \left(  d_{1}+d_{2}\right)  \left[  X,Y\right]  +\frac
{1}{2}\left(  d_{1}+d_{2}\right)  ^{2}\left[  \left[  X,Y\right]  ,Y\right]
\right)  -\\
\frac{1}{6}\left(  d_{1}+d_{2}\right)  ^{2}\left[  X+Y,\left[  X,Y\right]
\right]
\end{array}
\right\}  \\
&  =\exp\,d_{3}X.\exp\,\left(  d_{1}+d_{2}\right)  \left(  X+Y\right)
+\frac{1}{2}\left(  d_{1}+d_{2}\right)  ^{2}\left[  X,Y\right]  +d_{3}Y.\\
&  \exp\,d_{3}\left\{
\begin{array}
[c]{c}%
\frac{1}{2}\left(  \left(  d_{1}+d_{2}\right)  \left[  X,Y\right]  +\frac
{1}{2}\left(  d_{1}+d_{2}\right)  ^{2}\left[  \left[  X,Y\right]  ,Y\right]
\right)  -\\
\frac{1}{6}\left(  d_{1}+d_{2}\right)  ^{2}\left[  X+Y,\left[  X,Y\right]
\right]
\end{array}
\right\}  \\
&  \left)
\begin{array}
[c]{c}%
\text{By Theorem \ref{t5.2}\ with}\\%
\begin{array}
[c]{c}%
\delta^{\mathrm{left}}\left(  \exp\right)  \left(  \left(  d_{1}+d_{2}\right)
\left(  X+Y\right)  +\frac{1}{2}\left(  d_{1}+d_{2}\right)  ^{2}\left[
X,Y\right]  \right)  \left(  Y\right)  \\
=Y-\frac{1}{2}\left(  \left(  d_{1}+d_{2}\right)  \left[  X,Y\right]
+\frac{1}{2}\left(  d_{1}+d_{2}\right)  ^{2}\left[  \left[  X,Y\right]
,Y\right]  \right)  +\\
\frac{1}{6}\left(  d_{1}+d_{2}\right)  ^{2}\left[  X+Y,\left[  X,Y\right]
\right]
\end{array}
\end{array}
\right(
\end{align*}
We keep on.
\begin{align*}
&  =\exp\,d_{3}X.\exp\,\left(  d_{1}+d_{2}\right)  \left(  X+Y\right)
+\frac{1}{2}\left(  d_{1}+d_{2}\right)  ^{2}\left[  X,Y\right]  +d_{3}Y.\\
&  \exp\,d_{3}\left\{  \frac{1}{2}\left(  d_{1}+d_{2}\right)  \left[
X,Y\right]  -\frac{1}{4}\left(  d_{1}+d_{2}\right)  ^{2}\left[  X+Y,\left[
X,Y\right]  \right]  \right\}  .\\
&  \exp\,d_{3}\left\{
\begin{array}
[c]{c}%
\frac{1}{4}\left(  d_{1}+d_{2}\right)  ^{2}\left[  \left[  X,Y\right]
,Y\right]  -\frac{1}{6}\left(  d_{1}+d_{2}\right)  ^{2}\left[  X+Y,\left[
X,Y\right]  \right]  +\\
\frac{1}{4}\left(  d_{1}+d_{2}\right)  ^{2}\left[  X+Y,\left[  X,Y\right]
\right]
\end{array}
\right\}  \\
&  =\exp\,d_{3}X.\\
&  \exp\,\left(  d_{1}+d_{2}\right)  \left(  X+Y\right)  +\frac{1}{2}\left(
d_{1}+d_{2}\right)  ^{2}\left[  X,Y\right]  +d_{3}\left\{  Y+\frac{1}%
{2}\left(  d_{1}+d_{2}\right)  \left[  X,Y\right]  \right\}  .\\
&  \exp\,d_{3}\left\{
\begin{array}
[c]{c}%
\frac{1}{4}\left(  d_{1}+d_{2}\right)  ^{2}\left[  \left[  X,Y\right]
,Y\right]  -\frac{1}{6}\left(  d_{1}+d_{2}\right)  ^{2}\left[  X+Y,\left[
X,Y\right]  \right]  +\\
\frac{1}{4}\left(  d_{1}+d_{2}\right)  ^{2}\left[  X+Y,\left[  X,Y\right]
\right]
\end{array}
\right\}  \\
&  \left)
\begin{array}
[c]{c}%
\text{By Theorem \ref{t5.2}\ with }\\
\delta^{\mathrm{left}}\left(  \exp\right)  \left(  \left(  d_{1}+d_{2}\right)
\left(  X+Y\right)  +\frac{1}{2}\left(  d_{1}+d_{2}\right)  ^{2}\left[
X,Y\right]  +d_{3}Y\right)  \\
\left(  \frac{1}{2}\left(  d_{1}+d_{2}\right)  \left[  X,Y\right]  \right)  \\
=\frac{1}{2}\left(  d_{1}+d_{2}\right)  \left[  X,Y\right]  -\frac{1}%
{4}\left(  d_{1}+d_{2}\right)  ^{2}\left[  X+Y,\left[  X,Y\right]  \right]
\end{array}
\right(
\end{align*}
We keep on again.
\begin{align*}
&  =\exp\,d_{3}\left\{
\begin{array}
[c]{c}%
-\frac{1}{2}\left(  \left(  d_{1}+d_{2}\right)  \left[  Y,X\right]  +\frac
{1}{2}\left(  d_{1}+d_{2}\right)  ^{2}\left[  \left[  X,Y\right]  ,X\right]
\right)  -\\
\frac{1}{6}\left(  d_{1}+d_{2}\right)  ^{2}\left[  X+Y,\left[  Y,X\right]
\right]
\end{array}
\right\}  .\\
&  \exp\,d_{3}\left\{
\begin{array}
[c]{c}%
X+\frac{1}{2}\left(  \left(  d_{1}+d_{2}\right)  \left[  Y,X\right]  +\frac
{1}{2}\left(  d_{1}+d_{2}\right)  ^{2}\left[  \left[  X,Y\right]  ,X\right]
\right)  +\\
\frac{1}{6}\left(  d_{1}+d_{2}\right)  ^{2}\left[  X+Y,\left[  Y,X\right]
\right]
\end{array}
\right\}  .\\
&  \exp\,\left(  d_{1}+d_{2}\right)  \left(  X+Y\right)  +\frac{1}{2}\left(
d_{1}+d_{2}\right)  ^{2}\left[  X,Y\right]  +d_{3}\left\{  Y+\frac{1}%
{2}\left(  d_{1}+d_{2}\right)  \left[  X,Y\right]  \right\}  .\\
&  \exp\,d_{3}\left\{
\begin{array}
[c]{c}%
\frac{1}{4}\left(  d_{1}+d_{2}\right)  ^{2}\left[  \left[  X,Y\right]
,Y\right]  -\frac{1}{6}\left(  d_{1}+d_{2}\right)  ^{2}\left[  X+Y,\left[
X,Y\right]  \right]  +\\
\frac{1}{4}\left(  d_{1}+d_{2}\right)  ^{2}\left[  X+Y,\left[  X,Y\right]
\right]
\end{array}
\right\}  \\
&  =\exp\,d_{3}\left\{
\begin{array}
[c]{c}%
-\frac{1}{2}\left(  \left(  d_{1}+d_{2}\right)  \left[  Y,X\right]  +\frac
{1}{2}\left(  d_{1}+d_{2}\right)  ^{2}\left[  \left[  X,Y\right]  ,X\right]
\right)  -\\
\frac{1}{6}\left(  d_{1}+d_{2}\right)  ^{2}\left[  X+Y,\left[  Y,X\right]
\right]
\end{array}
\right\}  .\\
&  \exp\,\left(  d_{1}+d_{2}\right)  \left(  X+Y\right)  +\frac{1}{2}\left(
d_{1}+d_{2}\right)  ^{2}\left[  X,Y\right]  +d_{3}\left\{  \left(  X+Y\right)
+\frac{1}{2}\left(  d_{1}+d_{2}\right)  \left[  X,Y\right]  \right\}  .\\
&  \exp\,d_{3}\left\{
\begin{array}
[c]{c}%
\frac{1}{4}\left(  d_{1}+d_{2}\right)  ^{2}\left[  \left[  X,Y\right]
,Y\right]  -\frac{1}{6}\left(  d_{1}+d_{2}\right)  ^{2}\left[  X+Y,\left[
X,Y\right]  \right]  +\\
\frac{1}{4}\left(  d_{1}+d_{2}\right)  ^{2}\left[  X+Y,\left[  X,Y\right]
\right]
\end{array}
\right\}  \\
&  \left)
\begin{array}
[c]{c}%
\text{By Theorem \ref{t5.6}\ with }\\
\delta^{\mathrm{right}}\left(  \exp\right)  \left(
\begin{array}
[c]{c}%
\left(  d_{1}+d_{2}\right)  \left(  X+Y\right)  +\frac{1}{2}\left(
d_{1}+d_{2}\right)  ^{2}\left[  X,Y\right]  +\\
d_{3}\left\{  Y+\frac{1}{2}\left(  d_{1}+d_{2}\right)  \left[  X,Y\right]
\right\}
\end{array}
\right)  \left(  X\right)  \\
=X+\frac{1}{2}\left(  \left(  d_{1}+d_{2}\right)  \left[  Y,X\right]
+\frac{1}{2}\left(  d_{1}+d_{2}\right)  ^{2}\left[  \left[  X,Y\right]
,X\right]  \right)  +\\
\frac{1}{6}\left(  d_{1}+d_{2}\right)  ^{2}\left[  X+Y,\left[  Y,X\right]
\right]
\end{array}
\right(  \\
&  =\exp\,d_{3}\left\{
\begin{array}
[c]{c}%
-\frac{1}{4}\left(  d_{1}+d_{2}\right)  ^{2}\left[  \left[  X,Y\right]
,X\right]  -\frac{1}{6}\left(  d_{1}+d_{2}\right)  ^{2}\left[  X+Y,\left[
Y,X\right]  \right]  +\\
\frac{1}{4}\left(  d_{1}+d_{2}\right)  ^{2}\left[  X+Y,\left[  Y,X\right]
\right]
\end{array}
\right\}  .\\
&  \exp\,d_{3}\left\{  -\frac{1}{2}\left(  d_{1}+d_{2}\right)  \left[
Y,X\right]  -\frac{1}{4}\left(  d_{1}+d_{2}\right)  ^{2}\left[  X+Y,\left[
Y,X\right]  \right]  \right\}  .\\
&  \exp\,\left(  d_{1}+d_{2}\right)  \left(  X+Y\right)  +\frac{1}{2}\left(
d_{1}+d_{2}\right)  ^{2}\left[  X,Y\right]  +d_{3}\left\{  \left(  X+Y\right)
+\frac{1}{2}\left(  d_{1}+d_{2}\right)  \left[  X,Y\right]  \right\}  .\\
&  \exp\,d_{3}\left\{
\begin{array}
[c]{c}%
\frac{1}{4}\left(  d_{1}+d_{2}\right)  ^{2}\left[  \left[  X,Y\right]
,Y\right]  -\frac{1}{6}\left(  d_{1}+d_{2}\right)  ^{2}\left[  X+Y,\left[
X,Y\right]  \right]  +\\
\frac{1}{4}\left(  d_{1}+d_{2}\right)  ^{2}\left[  X+Y,\left[  X,Y\right]
\right]
\end{array}
\right\}  \\
&  =\exp\,d_{3}\left\{
\begin{array}
[c]{c}%
-\frac{1}{4}\left(  d_{1}+d_{2}\right)  ^{2}\left[  \left[  X,Y\right]
,X\right]  -\frac{1}{6}\left(  d_{1}+d_{2}\right)  ^{2}\left[  X+Y,\left[
Y,X\right]  \right]  +\\
\frac{1}{4}\left(  d_{1}+d_{2}\right)  ^{2}\left[  X+Y,\left[  Y,X\right]
\right]
\end{array}
\right\}  .\\
&  \exp\,\left(  d_{1}+d_{2}+d_{3}\right)  \left(  X+Y\right)  +\frac{1}%
{2}\left(  d_{1}+d_{2}+d_{3}\right)  ^{2}\left[  X,Y\right]  .\\
&  \exp\,d_{3}\left\{
\begin{array}
[c]{c}%
\frac{1}{4}\left(  d_{1}+d_{2}\right)  ^{2}\left[  \left[  X,Y\right]
,Y\right]  -\frac{1}{6}\left(  d_{1}+d_{2}\right)  ^{2}\left[  X+Y,\left[
X,Y\right]  \right]  +\\
\frac{1}{4}\left(  d_{1}+d_{2}\right)  ^{2}\left[  X+Y,\left[  X,Y\right]
\right]
\end{array}
\right\}  \\
&  \left)
\begin{array}
[c]{c}%
\text{By Theorem \ref{t5.6}\ with }\\
\delta^{\mathrm{right}}\left(  \exp\right)  \left(
\begin{array}
[c]{c}%
\,\left(  d_{1}+d_{2}\right)  \left(  X+Y\right)  +\\
\frac{1}{2}\left(  d_{1}+d_{2}\right)  ^{2}\left[  X,Y\right]  +\\
d_{3}\left\{  \left(  X+Y\right)  +\frac{1}{2}\left(  d_{1}+d_{2}\right)
\left[  X,Y\right]  \right\}
\end{array}
\right)  \left(  -\frac{1}{2}\left(  d_{1}+d_{2}\right)  \left[  Y,X\right]
\right)  \\
=-\frac{1}{2}\left(  d_{1}+d_{2}\right)  \left[  Y,X\right]  -\frac{1}%
{4}\left(  d_{1}+d_{2}\right)  ^{2}\left[  X+Y,\left[  Y,X\right]  \right]
\end{array}
\right(  \\
&  =\exp\,\left(  d_{1}+d_{2}+d_{3}\right)  \left(  X+Y\right)  +\frac{1}%
{2}\left(  d_{1}+d_{2}+d_{3}\right)  ^{2}\left[  X,Y\right]  +\\
&  \frac{1}{4}\left(  d_{1}+d_{2}\right)  ^{2}d_{3}\left(  \left[  X,\left[
X,Y\right]  \right]  -\left[  Y,\left[  X,Y\right]  \right]  \right)  \\
&  =\exp\,\left(  d_{1}+d_{2}+d_{3}\right)  \left(  X+Y\right)  +\frac{1}%
{2}\left(  d_{1}+d_{2}+d_{3}\right)  ^{2}\left[  X,Y\right]  +\\
&  \frac{1}{12}\left(  d_{1}+d_{2}+d_{3}\right)  ^{3}\left[  X-Y,\left[
X,Y\right]  \right]
\end{align*}

\end{proof}

\begin{theorem}
\label{t8.4}Given $X,Y\in\mathfrak{g}$\ and $d_{1},d_{2},d_{3},d_{4}\in D$, we
have
\begin{align*}
&  \exp\,\left(  d_{1}+d_{2}+d_{3}+d_{4}\right)  X.\exp\,\left(  d_{1}%
+d_{2}+d_{3}+d_{4}\right)  Y\\
&  =\exp\,\left(  d_{1}+d_{2}+d_{3}+d_{4}\right)  \left(  X+Y\right)
+\frac{\left(  d_{1}+d_{2}+d_{3}+d_{4}\right)  ^{2}}{2}\left[  X,Y\right]  +\\
&  \frac{\left(  d_{1}+d_{2}+d_{3}+d_{4}\right)  ^{3}}{12}\left[  X-Y,\left[
X,Y\right]  \right]  -\\
&  \frac{\left(  d_{1}+d_{2}+d_{3}+d_{4}\right)  ^{4}}{48}\left(
\begin{array}
[c]{c}%
\left[  X,\left[  Y,\left[  X,Y\right]  \right]  \right]  +\left[  Y,\left[
X,\left[  X,Y\right]  \right]  \right]  +\\
\left[  X+Y,\left[  X+Y,\left[  X,Y\right]  \right]  \right]
\end{array}
\right)
\end{align*}

\end{theorem}

\begin{proof}
We have
\begin{align*}
&  \exp\,\left(  d_{1}+d_{2}+d_{3}+d_{4}\right)  X.\exp\,\left(  d_{1}%
+d_{2}+d_{3}+d_{4}\right)  Y\\
&  =\exp\,\left(  d_{1}+d_{2}+d_{3}\right)  X+d_{4}X.\exp\,\left(  d_{1}%
+d_{2}+d_{3}\right)  Y+d_{4}Y\\
&  =\exp\,d_{4}X.\exp\,\left(  d_{1}+d_{2}+d_{3}\right)  X.\exp\,\left(
d_{1}+d_{2}+d_{3}\right)  Y.\exp\,d_{4}Y\\
&  \left)  \text{By Proposition \ref{t5.3}}\right( \\
&  =\exp\,d_{4}X.\\
&  \exp\,\left(  d_{1}+d_{2}+d_{3}\right)  \left(  X+Y\right)  +\frac{1}%
{2}\left(  d_{1}+d_{2}+d_{3}\right)  ^{2}\left[  X,Y\right]  +\\
&  \frac{1}{12}\left(  d_{1}+d_{2}+d_{3}\right)  ^{3}\left[  X-Y,\left[
X,Y\right]  \right]  .\\
&  \exp\,d_{4}Y\\
&  \left)  \text{By Theorem \ref{t8.3}}\right( \\
&  =\exp\,d_{4}X.\\
&  \exp\,\left(  d_{1}+d_{2}+d_{3}\right)  \left(  X+Y\right)  +\frac{1}%
{2}\left(  d_{1}+d_{2}+d_{3}\right)  ^{2}\left[  X,Y\right]  +\\
&  \frac{1}{12}\left(  d_{1}+d_{2}+d_{3}\right)  ^{3}\left[  X-Y,\left[
X,Y\right]  \right]  .\\
&  \exp\,d_{4}\left\{
\begin{array}
[c]{c}%
Y-\frac{1}{2}\left(
\begin{array}
[c]{c}%
\left(  d_{1}+d_{2}+d_{3}\right)  \left[  X,Y\right]  +\frac{1}{2}\left(
d_{1}+d_{2}+d_{3}\right)  ^{2}\left[  \left[  X,Y\right]  ,Y\right]  +\\
\frac{1}{12}\left(  d_{1}+d_{2}+d_{3}\right)  ^{3}\left[  \left[  X-Y,\left[
X,Y\right]  \right]  ,Y\right]
\end{array}
\right)  +\\
\frac{1}{6}\left(
\begin{array}
[c]{c}%
\left(  d_{1}+d_{2}+d_{3}\right)  ^{2}\left[  X+Y,\left[  X,Y\right]  \right]
+\\
\frac{1}{2}\left(  d_{1}+d_{2}+d_{3}\right)  ^{3}\left[  X+Y,\left[  \left[
X,Y\right]  ,Y\right]  \right]
\end{array}
\right)
\end{array}
\right\}  .\\
&  \exp\,d_{4}\left\{
\begin{array}
[c]{c}%
\frac{1}{2}\left(
\begin{array}
[c]{c}%
\left(  d_{1}+d_{2}+d_{3}\right)  \left[  X,Y\right]  +\frac{1}{2}\left(
d_{1}+d_{2}+d_{3}\right)  ^{2}\left[  \left[  X,Y\right]  ,Y\right]  +\\
\frac{1}{12}\left(  d_{1}+d_{2}+d_{3}\right)  ^{3}\left[  \left[  X-Y,\left[
X,Y\right]  \right]  ,Y\right]
\end{array}
\right)  -\\
\frac{1}{6}\left(
\begin{array}
[c]{c}%
\left(  d_{1}+d_{2}+d_{3}\right)  ^{2}\left[  X+Y,\left[  X,Y\right]  \right]
+\\
\frac{1}{2}\left(  d_{1}+d_{2}+d_{3}\right)  ^{3}\left[  X+Y,\left[  \left[
X,Y\right]  ,Y\right]  \right]
\end{array}
\right)
\end{array}
\right\} \\
&  =\exp\,d_{4}X.\\
&  \exp\,\left(  d_{1}+d_{2}+d_{3}\right)  \left(  X+Y\right)  +\frac{1}%
{2}\left(  d_{1}+d_{2}+d_{3}\right)  ^{2}\left[  X,Y\right]  +\\
&  \frac{1}{12}\left(  d_{1}+d_{2}+d_{3}\right)  ^{3}\left[  X-Y,\left[
X,Y\right]  \right]  +d_{4}Y.\\
&  \exp\,d_{4}\left\{
\begin{array}
[c]{c}%
\frac{1}{2}\left(
\begin{array}
[c]{c}%
\left(  d_{1}+d_{2}+d_{3}\right)  \left[  X,Y\right]  +\frac{1}{2}\left(
d_{1}+d_{2}+d_{3}\right)  ^{2}\left[  \left[  X,Y\right]  ,Y\right]  +\\
\frac{1}{12}\left(  d_{1}+d_{2}+d_{3}\right)  ^{3}\left[  \left[  X-Y,\left[
X,Y\right]  \right]  ,Y\right]
\end{array}
\right)  -\\
\frac{1}{6}\left(
\begin{array}
[c]{c}%
\left(  d_{1}+d_{2}+d_{3}\right)  ^{2}\left[  X+Y,\left[  X,Y\right]  \right]
+\\
\frac{1}{2}\left(  d_{1}+d_{2}+d_{3}\right)  ^{3}\left[  X+Y,\left[  \left[
X,Y\right]  ,Y\right]  \right]
\end{array}
\right)
\end{array}
\right\} \\
&  \left)
\begin{array}
[c]{c}%
\text{By Theorem \ref{t5.2}\ with }\\
\delta^{\mathrm{left}}\left(  \exp\right)  \left(
\begin{array}
[c]{c}%
\left(  d_{1}+d_{2}+d_{3}\right)  \left(  X+Y\right)  +\frac{1}{2}\left(
d_{1}+d_{2}+d_{3}\right)  ^{2}\left[  X,Y\right]  +\\
\frac{1}{12}\left(  d_{1}+d_{2}+d_{3}\right)  ^{3}\left[  X-Y,\left[
X,Y\right]  \right]
\end{array}
\right)  \left(  Y\right) \\
=Y-\frac{1}{2}\left(
\begin{array}
[c]{c}%
\left(  d_{1}+d_{2}+d_{3}\right)  \left[  X,Y\right]  +\frac{1}{2}\left(
d_{1}+d_{2}+d_{3}\right)  ^{2}\left[  \left[  X,Y\right]  ,Y\right]  +\\
\frac{1}{12}\left(  d_{1}+d_{2}+d_{3}\right)  ^{3}\left[  \left[  X-Y,\left[
X,Y\right]  \right]  ,Y\right]
\end{array}
\right)  +\\
\frac{1}{6}\left(
\begin{array}
[c]{c}%
\left(  d_{1}+d_{2}+d_{3}\right)  ^{2}\left[  X+Y,\left[  X,Y\right]  \right]
+\\
\frac{1}{2}\left(  d_{1}+d_{2}+d_{3}\right)  ^{3}\left[  X+Y,\left[  \left[
X,Y\right]  ,Y\right]  \right]
\end{array}
\right)
\end{array}
\right(
\end{align*}
We keep on.
\begin{align*}
&  =\exp\,d_{4}X.\\
&  \exp\,\left(  d_{1}+d_{2}+d_{3}\right)  \left(  X+Y\right)  +\frac{1}%
{2}\left(  d_{1}+d_{2}+d_{3}\right)  ^{2}\left[  X,Y\right]  +\\
&  \frac{1}{12}\left(  d_{1}+d_{2}+d_{3}\right)  ^{3}\left[  X-Y,\left[
X,Y\right]  \right]  +d_{4}Y.\\
&  \exp\,d_{4}\left\{
\begin{array}
[c]{c}%
\frac{1}{2}\left(  d_{1}+d_{2}+d_{3}\right)  \left[  X,Y\right]  +\frac{1}%
{4}\left(  d_{1}+d_{2}+d_{3}\right)  ^{2}\left[  \left[  X,Y\right]
,Y\right]  -\\
\frac{1}{6}\left(  d_{1}+d_{2}+d_{3}\right)  ^{2}\left[  X+Y,\left[
X,Y\right]  \right]  -\\
\frac{1}{2}\left(
\begin{array}
[c]{c}%
\frac{1}{2}\left(  d_{1}+d_{2}+d_{3}\right)  ^{2}\left[  X+Y,\left[
X,Y\right]  \right]  +\\
\frac{1}{4}\left(  d_{1}+d_{2}+d_{3}\right)  ^{3}\left[  X+Y,\left[  \left[
X,Y\right]  ,Y\right]  \right]  -\\
\frac{1}{6}\left(  d_{1}+d_{2}+d_{3}\right)  ^{3}\left[  X+Y,\left[
X+Y,\left[  X,Y\right]  \right]  \right]
\end{array}
\right)  +\\
\frac{1}{12}\left(  d_{1}+d_{2}+d_{3}\right)  ^{3}\left[  X+Y,\left[
X+Y,\left[  X,Y\right]  \right]  \right]
\end{array}
\right\}  .\\
&  \exp\,d_{4}\left\{
\begin{array}
[c]{c}%
\frac{1}{24}\left(  d_{1}+d_{2}+d_{3}\right)  ^{3}\left[  \left[  X-Y,\left[
X,Y\right]  \right]  ,Y\right]  -\\
\frac{1}{12}\left(  d_{1}+d_{2}+d_{3}\right)  ^{3}\left[  X+Y,\left[  \left[
X,Y\right]  ,Y\right]  \right]  +\\
\frac{1}{2}\left(
\begin{array}
[c]{c}%
\frac{1}{2}\left(  d_{1}+d_{2}+d_{3}\right)  ^{2}\left[  X+Y,\left[
X,Y\right]  \right]  +\\
\frac{1}{4}\left(  d_{1}+d_{2}+d_{3}\right)  ^{3}\left[  X+Y,\left[  \left[
X,Y\right]  ,Y\right]  \right]  -\\
\frac{1}{6}\left(  d_{1}+d_{2}+d_{3}\right)  ^{3}\left[  X+Y,\left[
X+Y,\left[  X,Y\right]  \right]  \right]
\end{array}
\right)  -\\
\frac{1}{12}\left(  d_{1}+d_{2}+d_{3}\right)  ^{3}\left[  X+Y,\left[
X+Y,\left[  X,Y\right]  \right]  \right]
\end{array}
\right\} \\
&  =\exp\,d_{4}X.\\
&  \exp\,\left(  d_{1}+d_{2}+d_{3}\right)  \left(  X+Y\right)  +\frac{1}%
{2}\left(  d_{1}+d_{2}+d_{3}\right)  ^{2}\left[  X,Y\right]  +\\
&  \frac{1}{12}\left(  d_{1}+d_{2}+d_{3}\right)  ^{3}\left[  X-Y,\left[
X,Y\right]  \right]  +\\
&  d_{4}\left\{
\begin{array}
[c]{c}%
Y+\frac{1}{2}\left(  d_{1}+d_{2}+d_{3}\right)  \left[  X,Y\right]  +\frac
{1}{4}\left(  d_{1}+d_{2}+d_{3}\right)  ^{2}\left[  \left[  X,Y\right]
,Y\right]  -\\
\frac{1}{6}\left(  d_{1}+d_{2}+d_{3}\right)  ^{2}\left[  X+Y,\left[
X,Y\right]  \right]
\end{array}
\right\}  .\\
&  \exp\,d_{4}\left\{
\begin{array}
[c]{c}%
\frac{1}{24}\left(  d_{1}+d_{2}+d_{3}\right)  ^{3}\left[  \left[  X-Y,\left[
X,Y\right]  \right]  ,Y\right]  -\\
\frac{1}{12}\left(  d_{1}+d_{2}+d_{3}\right)  ^{3}\left[  X+Y,\left[  \left[
X,Y\right]  ,Y\right]  \right]  +\\
\frac{1}{2}\left(
\begin{array}
[c]{c}%
\frac{1}{2}\left(  d_{1}+d_{2}+d_{3}\right)  ^{2}\left[  X+Y,\left[
X,Y\right]  \right]  +\\
\frac{1}{4}\left(  d_{1}+d_{2}+d_{3}\right)  ^{3}\left[  X+Y,\left[  \left[
X,Y\right]  ,Y\right]  \right]  -\\
\frac{1}{6}\left(  d_{1}+d_{2}+d_{3}\right)  ^{3}\left[  X+Y,\left[
X+Y,\left[  X,Y\right]  \right]  \right]
\end{array}
\right)  -\\
\frac{1}{12}\left(  d_{1}+d_{2}+d_{3}\right)  ^{3}\left[  X+Y,\left[
X+Y,\left[  X,Y\right]  \right]  \right]
\end{array}
\right\} \\
&  \left)
\begin{array}
[c]{c}%
\text{By Theorem \ref{t5.2}\ with }\\
\delta^{\mathrm{left}}\left(  \exp\right)  \left(
\begin{array}
[c]{c}%
\left(  d_{1}+d_{2}+d_{3}\right)  \left(  X+Y\right)  +\\
\frac{1}{2}\left(  d_{1}+d_{2}+d_{3}\right)  ^{2}\left[  X,Y\right]  +\\
\frac{1}{12}\left(  d_{1}+d_{2}+d_{3}\right)  ^{3}\left[  X-Y,\left[
X,Y\right]  \right]  +\\
d_{4}Y
\end{array}
\right) \\
\left(
\begin{array}
[c]{c}%
\frac{1}{2}\left(  d_{1}+d_{2}+d_{3}\right)  \left[  X,Y\right]  +\\
\frac{1}{4}\left(  d_{1}+d_{2}+d_{3}\right)  ^{2}\left[  \left[  X,Y\right]
,Y\right]  -\\
\frac{1}{6}\left(  d_{1}+d_{2}+d_{3}\right)  ^{2}\left[  X+Y,\left[
X,Y\right]  \right]
\end{array}
\right) \\
=\frac{1}{2}\left(  d_{1}+d_{2}+d_{3}\right)  \left[  X,Y\right]  +\frac{1}%
{4}\left(  d_{1}+d_{2}+d_{3}\right)  ^{2}\left[  \left[  X,Y\right]
,Y\right]  -\\
\frac{1}{6}\left(  d_{1}+d_{2}+d_{3}\right)  ^{2}\left[  X+Y,\left[
X,Y\right]  \right]  -\\
\frac{1}{2}\left(
\begin{array}
[c]{c}%
\frac{1}{2}\left(  d_{1}+d_{2}+d_{3}\right)  ^{2}\left[  X+Y,\left[
X,Y\right]  \right]  +\\
\frac{1}{4}\left(  d_{1}+d_{2}+d_{3}\right)  ^{3}\left[  X+Y,\left[  \left[
X,Y\right]  ,Y\right]  \right]  -\\
\frac{1}{6}\left(  d_{1}+d_{2}+d_{3}\right)  ^{3}\left[  X+Y,\left[
X+Y,\left[  X,Y\right]  \right]  \right]
\end{array}
\right)  +\\
\frac{1}{12}\left(  d_{1}+d_{2}+d_{3}\right)  ^{3}\left[  X+Y,\left[
X+Y,\left[  X,Y\right]  \right]  \right]
\end{array}
\right(
\end{align*}
We keep on again.
\begin{align*}
&  =\exp\,d_{4}X.\\
&  \exp\,\left(  d_{1}+d_{2}+d_{3}\right)  \left(  X+Y\right)  +\frac{1}%
{2}\left(  d_{1}+d_{2}+d_{3}\right)  ^{2}\left[  X,Y\right]  +\\
&  \frac{1}{12}\left(  d_{1}+d_{2}+d_{3}\right)  ^{3}\left[  X-Y,\left[
X,Y\right]  \right]  +\\
&  d_{4}\left\{
\begin{array}
[c]{c}%
Y+\frac{1}{2}\left(  d_{1}+d_{2}+d_{3}\right)  \left[  X,Y\right]  +\frac
{1}{4}\left(  d_{1}+d_{2}+d_{3}\right)  ^{2}\left[  \left[  X,Y\right]
,Y\right]  -\\
\frac{1}{6}\left(  d_{1}+d_{2}+d_{3}\right)  ^{2}\left[  X+Y,\left[
X,Y\right]  \right]
\end{array}
\right\}  .\\
&  \exp\,d_{4}\left\{
\begin{array}
[c]{c}%
\frac{1}{4}\left(  d_{1}+d_{2}+d_{3}\right)  ^{2}\left[  X+Y,\left[
X,Y\right]  \right]  -\\
\frac{1}{8}\left(  d_{1}+d_{2}+d_{3}\right)  ^{3}\left[  X+Y,\left[
X+Y,\left[  X,Y\right]  \right]  \right]
\end{array}
\right\}  .\\
&  \exp\,d_{4}\left\{
\begin{array}
[c]{c}%
\frac{1}{24}\left(  d_{1}+d_{2}+d_{3}\right)  ^{3}\left[  \left[  X-Y,\left[
X,Y\right]  \right]  ,Y\right]  -\\
\frac{1}{12}\left(  d_{1}+d_{2}+d_{3}\right)  ^{3}\left[  X+Y,\left[  \left[
X,Y\right]  ,Y\right]  \right]  +\\
\frac{1}{2}\left(
\begin{array}
[c]{c}%
\frac{1}{4}\left(  d_{1}+d_{2}+d_{3}\right)  ^{3}\left[  X+Y,\left[  \left[
X,Y\right]  ,Y\right]  \right]  -\\
\frac{1}{6}\left(  d_{1}+d_{2}+d_{3}\right)  ^{3}\left[  X+Y,\left[
X+Y,\left[  X,Y\right]  \right]  \right]
\end{array}
\right)  -\\
\frac{1}{12}\left(  d_{1}+d_{2}+d_{3}\right)  ^{3}\left[  X+Y,\left[
X+Y,\left[  X,Y\right]  \right]  \right]  +\\
\frac{1}{8}\left(  d_{1}+d_{2}+d_{3}\right)  ^{3}\left[  X+Y,\left[
X+Y,\left[  X,Y\right]  \right]  \right]
\end{array}
\right\} \\
&  =\exp\,d_{4}X.\\
&  \exp\,\left(  d_{1}+d_{2}+d_{3}\right)  \left(  X+Y\right)  +\frac{1}%
{2}\left(  d_{1}+d_{2}+d_{3}\right)  ^{2}\left[  X,Y\right]  +\\
&  \frac{1}{12}\left(  d_{1}+d_{2}+d_{3}\right)  ^{3}\left[  X-Y,\left[
X,Y\right]  \right]  +\\
&  d_{4}\left\{
\begin{array}
[c]{c}%
Y+\frac{1}{2}\left(  d_{1}+d_{2}+d_{3}\right)  \left[  X,Y\right]  +\frac
{1}{4}\left(  d_{1}+d_{2}+d_{3}\right)  ^{2}\left[  \left[  X,Y\right]
,Y\right]  -\\
\frac{1}{6}\left(  d_{1}+d_{2}+d_{3}\right)  ^{2}\left[  X+Y,\left[
X,Y\right]  \right]  +\frac{1}{4}\left(  d_{1}+d_{2}+d_{3}\right)  ^{2}\left[
X+Y,\left[  X,Y\right]  \right]
\end{array}
\right\}  .\\
&  \exp\,d_{4}\left\{
\begin{array}
[c]{c}%
\frac{1}{24}\left(  d_{1}+d_{2}+d_{3}\right)  ^{3}\left[  \left[  X-Y,\left[
X,Y\right]  \right]  ,Y\right]  -\\
\frac{1}{12}\left(  d_{1}+d_{2}+d_{3}\right)  ^{3}\left[  X+Y,\left[  \left[
X,Y\right]  ,Y\right]  \right]  +\\
\frac{1}{2}\left(
\begin{array}
[c]{c}%
\frac{1}{4}\left(  d_{1}+d_{2}+d_{3}\right)  ^{3}\left[  X+Y,\left[  \left[
X,Y\right]  ,Y\right]  \right]  -\\
\frac{1}{6}\left(  d_{1}+d_{2}+d_{3}\right)  ^{3}\left[  X+Y,\left[
X+Y,\left[  X,Y\right]  \right]  \right]
\end{array}
\right)  -\\
\frac{1}{12}\left(  d_{1}+d_{2}+d_{3}\right)  ^{3}\left[  X+Y,\left[
X+Y,\left[  X,Y\right]  \right]  \right]  +\\
\frac{1}{8}\left(  d_{1}+d_{2}+d_{3}\right)  ^{3}\left[  X+Y,\left[
X+Y,\left[  X,Y\right]  \right]  \right]
\end{array}
\right\} \\
&  \left)
\begin{array}
[c]{c}%
\text{By Theorem \ref{t5.2}\ with }\\
\delta^{\mathrm{left}}\left(  \exp\right)  \left(
\begin{array}
[c]{c}%
\,\left(  d_{1}+d_{2}+d_{3}\right)  \left(  X+Y\right)  +\\
\frac{1}{2}\left(  d_{1}+d_{2}+d_{3}\right)  ^{2}\left[  X,Y\right]  +\\
\frac{1}{12}\left(  d_{1}+d_{2}+d_{3}\right)  ^{3}\left[  X-Y,\left[
X,Y\right]  \right]  +\\
d_{4}\left\{
\begin{array}
[c]{c}%
Y+\frac{1}{2}\left(  d_{1}+d_{2}+d_{3}\right)  \left[  X,Y\right]  +\\
\frac{1}{4}\left(  d_{1}+d_{2}+d_{3}\right)  ^{2}\left[  \left[  X,Y\right]
,Y\right]  -\\
\frac{1}{6}\left(  d_{1}+d_{2}+d_{3}\right)  ^{2}\left[  X+Y,\left[
X,Y\right]  \right]
\end{array}
\right\}
\end{array}
\right) \\
\left(  \frac{1}{4}\left(  d_{1}+d_{2}+d_{3}\right)  ^{2}\left[  X+Y,\left[
X,Y\right]  \right]  \right) \\
=\frac{1}{4}\left(  d_{1}+d_{2}+d_{3}\right)  ^{2}\left[  X+Y,\left[
X,Y\right]  \right]  -\\
\frac{1}{8}\left(  d_{1}+d_{2}+d_{3}\right)  ^{3}\left[  X+Y,\left[
X+Y,\left[  X,Y\right]  \right]  \right]
\end{array}
\right(
\end{align*}

We keep on once more.
\begin{align*}
&  =\exp\,d_{4}\left\{
\begin{array}
[c]{c}%
-\frac{1}{2}\left(
\begin{array}
[c]{c}%
\left(  d_{1}+d_{2}+d_{3}\right)  \left[  Y,X\right]  +\frac{1}{2}\left(
d_{1}+d_{2}+d_{3}\right)  ^{2}\left[  \left[  X,Y\right]  ,X\right]  +\\
\frac{1}{12}\left(  d_{1}+d_{2}+d_{3}\right)  ^{3}\left[  \left[  X-Y,\left[
X,Y\right]  \right]  ,X\right]
\end{array}
\right)  -\\
\frac{1}{6}\left(
\begin{array}
[c]{c}%
\left(  d_{1}+d_{2}+d_{3}\right)  ^{2}\left[  X+Y,\left[  Y,X\right]  \right]
+\\
\frac{1}{2}\left(  d_{1}+d_{2}+d_{3}\right)  ^{3}\left[  X+Y,\left[  \left[
X,Y\right]  ,X\right]  \right]
\end{array}
\right)
\end{array}
\right\}  .\\
&  \exp\,d_{4}\left\{
\begin{array}
[c]{c}%
X+\frac{1}{2}\left(
\begin{array}
[c]{c}%
\left(  d_{1}+d_{2}+d_{3}\right)  \left[  Y,X\right]  +\frac{1}{2}\left(
d_{1}+d_{2}+d_{3}\right)  ^{2}\left[  \left[  X,Y\right]  ,X\right]  +\\
\frac{1}{12}\left(  d_{1}+d_{2}+d_{3}\right)  ^{3}\left[  \left[  X-Y,\left[
X,Y\right]  \right]  ,X\right]
\end{array}
\right)  +\\
\frac{1}{6}\left(
\begin{array}
[c]{c}%
\left(  d_{1}+d_{2}+d_{3}\right)  ^{2}\left[  X+Y,\left[  Y,X\right]  \right]
+\\
\frac{1}{2}\left(  d_{1}+d_{2}+d_{3}\right)  ^{3}\left[  X+Y,\left[  \left[
X,Y\right]  ,X\right]  \right]
\end{array}
\right)
\end{array}
\right\}  .\\
&  \exp\,\left(  d_{1}+d_{2}+d_{3}\right)  \left(  X+Y\right)  +\frac{1}%
{2}\left(  d_{1}+d_{2}+d_{3}\right)  ^{2}\left[  X,Y\right]  +\\
&  \frac{1}{12}\left(  d_{1}+d_{2}+d_{3}\right)  ^{3}\left[  X-Y,\left[
X,Y\right]  \right]  +\\
&  d_{4}\left\{
\begin{array}
[c]{c}%
Y+\frac{1}{2}\left(  d_{1}+d_{2}+d_{3}\right)  \left[  X,Y\right]  +\frac
{1}{4}\left(  d_{1}+d_{2}+d_{3}\right)  ^{2}\left[  \left[  X,Y\right]
,Y\right]  -\\
\frac{1}{6}\left(  d_{1}+d_{2}+d_{3}\right)  ^{2}\left[  X+Y,\left[
X,Y\right]  \right]  +\frac{1}{4}\left(  d_{1}+d_{2}+d_{3}\right)  ^{2}\left[
X+Y,\left[  X,Y\right]  \right]
\end{array}
\right\}  .\\
&  \exp\,d_{4}\left\{
\begin{array}
[c]{c}%
\frac{1}{24}\left(  d_{1}+d_{2}+d_{3}\right)  ^{3}\left[  \left[  X-Y,\left[
X,Y\right]  \right]  ,Y\right]  -\\
\frac{1}{12}\left(  d_{1}+d_{2}+d_{3}\right)  ^{3}\left[  X+Y,\left[  \left[
X,Y\right]  ,Y\right]  \right]  +\\
\frac{1}{2}\left(
\begin{array}
[c]{c}%
\frac{1}{4}\left(  d_{1}+d_{2}+d_{3}\right)  ^{3}\left[  X+Y,\left[  \left[
X,Y\right]  ,Y\right]  \right]  -\\
\frac{1}{6}\left(  d_{1}+d_{2}+d_{3}\right)  ^{3}\left[  X+Y,\left[
X+Y,\left[  X,Y\right]  \right]  \right]
\end{array}
\right)  -\\
\frac{1}{12}\left(  d_{1}+d_{2}+d_{3}\right)  ^{3}\left[  X+Y,\left[
X+Y,\left[  X,Y\right]  \right]  \right]  +\\
\frac{1}{8}\left(  d_{1}+d_{2}+d_{3}\right)  ^{3}\left[  X+Y,\left[
X+Y,\left[  X,Y\right]  \right]  \right]
\end{array}
\right\} \\
&  =\exp\,d_{4}\left\{
\begin{array}
[c]{c}%
-\frac{1}{2}\left(
\begin{array}
[c]{c}%
\left(  d_{1}+d_{2}+d_{3}\right)  \left[  Y,X\right]  +\frac{1}{2}\left(
d_{1}+d_{2}+d_{3}\right)  ^{2}\left[  \left[  X,Y\right]  ,X\right]  +\\
\frac{1}{12}\left(  d_{1}+d_{2}+d_{3}\right)  ^{3}\left[  \left[  X-Y,\left[
X,Y\right]  \right]  ,X\right]
\end{array}
\right)  -\\
\frac{1}{6}\left(
\begin{array}
[c]{c}%
\left(  d_{1}+d_{2}+d_{3}\right)  ^{2}\left[  X+Y,\left[  Y,X\right]  \right]
+\\
\frac{1}{2}\left(  d_{1}+d_{2}+d_{3}\right)  ^{3}\left[  X+Y,\left[  \left[
X,Y\right]  ,X\right]  \right]
\end{array}
\right)
\end{array}
\right\}  .\\
&  \exp\,\left(  d_{1}+d_{2}+d_{3}\right)  \left(  X+Y\right)  +\frac{1}%
{2}\left(  d_{1}+d_{2}+d_{3}\right)  ^{2}\left[  X,Y\right]  +\\
&  \frac{1}{12}\left(  d_{1}+d_{2}+d_{3}\right)  ^{3}\left[  X-Y,\left[
X,Y\right]  \right]  +\\
&  d_{4}\left\{
\begin{array}
[c]{c}%
X+Y+\frac{1}{2}\left(  d_{1}+d_{2}+d_{3}\right)  \left[  X,Y\right]  +\frac
{1}{4}\left(  d_{1}+d_{2}+d_{3}\right)  ^{2}\left[  \left[  X,Y\right]
,Y\right]  -\\
\frac{1}{6}\left(  d_{1}+d_{2}+d_{3}\right)  ^{2}\left[  X+Y,\left[
X,Y\right]  \right]  +\\
\frac{1}{4}\left(  d_{1}+d_{2}+d_{3}\right)  ^{2}\left[  X+Y,\left[
X,Y\right]  \right]
\end{array}
\right\}  .\\
&  \exp\,d_{4}\left\{
\begin{array}
[c]{c}%
\frac{1}{24}\left(  d_{1}+d_{2}+d_{3}\right)  ^{3}\left[  \left[  X-Y,\left[
X,Y\right]  \right]  ,Y\right]  -\\
\frac{1}{12}\left(  d_{1}+d_{2}+d_{3}\right)  ^{3}\left[  X+Y,\left[  \left[
X,Y\right]  ,Y\right]  \right]  +\\
\frac{1}{2}\left(
\begin{array}
[c]{c}%
\frac{1}{4}\left(  d_{1}+d_{2}+d_{3}\right)  ^{3}\left[  X+Y,\left[  \left[
X,Y\right]  ,Y\right]  \right]  -\\
\frac{1}{6}\left(  d_{1}+d_{2}+d_{3}\right)  ^{3}\left[  X+Y,\left[
X+Y,\left[  X,Y\right]  \right]  \right]
\end{array}
\right)  -\\
\frac{1}{12}\left(  d_{1}+d_{2}+d_{3}\right)  ^{3}\left[  X+Y,\left[
X+Y,\left[  X,Y\right]  \right]  \right]  +\\
\frac{1}{8}\left(  d_{1}+d_{2}+d_{3}\right)  ^{3}\left[  X+Y,\left[
X+Y,\left[  X,Y\right]  \right]  \right]
\end{array}
\right\} \\
&  \left)
\begin{array}
[c]{c}%
\text{By Theorem \ref{t5.6}\ with }\\
\delta^{\mathrm{rightt}}\left(  \exp\right)  \left(
\begin{array}
[c]{c}%
\left(  d_{1}+d_{2}+d_{3}\right)  \left(  X+Y\right)  +\\
\frac{1}{2}\left(  d_{1}+d_{2}+d_{3}\right)  ^{2}\left[  X,Y\right]  +\\
\frac{1}{12}\left(  d_{1}+d_{2}+d_{3}\right)  ^{3}\left[  X-Y,\left[
X,Y\right]  \right]  +\\
d_{4}\left\{
\begin{array}
[c]{c}%
Y+\frac{1}{2}\left(  d_{1}+d_{2}+d_{3}\right)  \left[  X,Y\right]  +\\
\frac{1}{4}\left(  d_{1}+d_{2}+d_{3}\right)  ^{2}\left[  \left[  X,Y\right]
,Y\right]  -\\
\frac{1}{6}\left(  d_{1}+d_{2}+d_{3}\right)  ^{2}\left[  X+Y,\left[
X,Y\right]  \right]  +\\
\frac{1}{4}\left(  d_{1}+d_{2}+d_{3}\right)  ^{2}\left[  X+Y,\left[
X,Y\right]  \right]
\end{array}
\right\}
\end{array}
\right)  \left(  X\right) \\
=X+\frac{1}{2}\left(
\begin{array}
[c]{c}%
\left(  d_{1}+d_{2}+d_{3}\right)  \left[  Y,X\right]  +\frac{1}{2}\left(
d_{1}+d_{2}+d_{3}\right)  ^{2}\left[  \left[  X,Y\right]  ,X\right]  +\\
\frac{1}{12}\left(  d_{1}+d_{2}+d_{3}\right)  ^{3}\left[  \left[  X-Y,\left[
X,Y\right]  \right]  ,X\right]
\end{array}
\right)  +\\
\frac{1}{6}\left(
\begin{array}
[c]{c}%
\left(  d_{1}+d_{2}+d_{3}\right)  ^{2}\left[  X+Y,\left[  Y,X\right]  \right]
+\\
\frac{1}{2}\left(  d_{1}+d_{2}+d_{3}\right)  ^{3}\left[  X+Y,\left[  \left[
X,Y\right]  ,X\right]  \right]
\end{array}
\right)
\end{array}
\right(
\end{align*}
We keep on once more.
\begin{align*}
&  =\exp\,d_{4}\left\{
\begin{array}
[c]{c}%
-\frac{1}{24}\left(  d_{1}+d_{2}+d_{3}\right)  ^{3}\left[  \left[  X-Y,\left[
X,Y\right]  \right]  ,X\right]  -\\
\frac{1}{12}\left(  d_{1}+d_{2}+d_{3}\right)  ^{3}\left[  X+Y,\left[  \left[
X,Y\right]  ,X\right]  \right]  -\\
\frac{1}{2}\left(
\begin{array}
[c]{c}%
-\frac{1}{2}\left(  d_{1}+d_{2}+d_{3}\right)  ^{2}\left[  X+Y,\left[
Y,X\right]  \right]  -\\
\frac{1}{4}\left(  d_{1}+d_{2}+d_{3}\right)  ^{3}\left[  X+Y,\left[  \left[
X,Y\right]  ,X\right]  \right]  -\\
\frac{1}{6}\left(  d_{1}+d_{2}+d_{3}\right)  ^{3}\left[  X+Y,\left[
X+Y,\left[  Y,X\right]  \right]  \right]
\end{array}
\right)  +\\
\frac{1}{12}\left(  d_{1}+d_{2}+d_{3}\right)  ^{3}\left[  X+Y,\left[
X+Y,\left[  Y,X\right]  \right]  \right]
\end{array}
\right\}  .\\
&  \exp\,d_{4}\left\{
\begin{array}
[c]{c}%
-\frac{1}{2}\left(  d_{1}+d_{2}+d_{3}\right)  \left[  Y,X\right]  -\frac{1}%
{4}\left(  d_{1}+d_{2}+d_{3}\right)  ^{2}\left[  \left[  X,Y\right]
,X\right]  -\\
\frac{1}{6}\left(  d_{1}+d_{2}+d_{3}\right)  ^{2}\left[  X+Y,\left[
Y,X\right]  \right]  +\\
\frac{1}{2}\left(
\begin{array}
[c]{c}%
-\frac{1}{2}\left(  d_{1}+d_{2}+d_{3}\right)  ^{2}\left[  X+Y,\left[
Y,X\right]  \right]  -\\
\frac{1}{4}\left(  d_{1}+d_{2}+d_{3}\right)  ^{3}\left[  X+Y,\left[  \left[
X,Y\right]  ,X\right]  \right]  -\\
\frac{1}{6}\left(  d_{1}+d_{2}+d_{3}\right)  ^{3}\left[  X+Y,\left[
X+Y,\left[  Y,X\right]  \right]  \right]
\end{array}
\right)  -\\
\frac{1}{12}\left(  d_{1}+d_{2}+d_{3}\right)  ^{3}\left[  X+Y,\left[
X+Y,\left[  Y,X\right]  \right]  \right]
\end{array}
\right\}  .\\
&  \exp\,\left(  d_{1}+d_{2}+d_{3}\right)  \left(  X+Y\right)  +\frac{1}%
{2}\left(  d_{1}+d_{2}+d_{3}\right)  ^{2}\left[  X,Y\right]  +\\
&  \frac{1}{12}\left(  d_{1}+d_{2}+d_{3}\right)  ^{3}\left[  X-Y,\left[
X,Y\right]  \right]  +\\
&  d_{4}\left\{
\begin{array}
[c]{c}%
X+Y+\frac{1}{2}\left(  d_{1}+d_{2}+d_{3}\right)  \left[  X,Y\right]  +\frac
{1}{4}\left(  d_{1}+d_{2}+d_{3}\right)  ^{2}\left[  \left[  X,Y\right]
,Y\right]  -\\
\frac{1}{6}\left(  d_{1}+d_{2}+d_{3}\right)  ^{2}\left[  X+Y,\left[
X,Y\right]  \right]  +\\
\frac{1}{4}\left(  d_{1}+d_{2}+d_{3}\right)  ^{2}\left[  X+Y,\left[
X,Y\right]  \right]
\end{array}
\right\}  .\\
&  \exp\,d_{4}\left\{
\begin{array}
[c]{c}%
\frac{1}{24}\left(  d_{1}+d_{2}+d_{3}\right)  ^{3}\left[  \left[  X-Y,\left[
X,Y\right]  \right]  ,Y\right]  -\\
\frac{1}{12}\left(  d_{1}+d_{2}+d_{3}\right)  ^{3}\left[  X+Y,\left[  \left[
X,Y\right]  ,Y\right]  \right]  +\\
\frac{1}{2}\left(
\begin{array}
[c]{c}%
\frac{1}{4}\left(  d_{1}+d_{2}+d_{3}\right)  ^{3}\left[  X+Y,\left[  \left[
X,Y\right]  ,Y\right]  \right]  -\\
\frac{1}{6}\left(  d_{1}+d_{2}+d_{3}\right)  ^{3}\left[  X+Y,\left[
X+Y,\left[  X,Y\right]  \right]  \right]
\end{array}
\right)  -\\
\frac{1}{12}\left(  d_{1}+d_{2}+d_{3}\right)  ^{3}\left[  X+Y,\left[
X+Y,\left[  X,Y\right]  \right]  \right]  +\\
\frac{1}{8}\left(  d_{1}+d_{2}+d_{3}\right)  ^{3}\left[  X+Y,\left[
X+Y,\left[  X,Y\right]  \right]  \right]
\end{array}
\right\}
\end{align*}
We keep on once more.
\begin{align*}
&  =\exp\,d_{4}\left\{
\begin{array}
[c]{c}%
-\frac{1}{24}\left(  d_{1}+d_{2}+d_{3}\right)  ^{3}\left[  \left[  X-Y,\left[
X,Y\right]  \right]  ,X\right]  -\\
\frac{1}{12}\left(  d_{1}+d_{2}+d_{3}\right)  ^{3}\left[  X+Y,\left[  \left[
X,Y\right]  ,X\right]  \right]  -\\
\frac{1}{2}\left(
\begin{array}
[c]{c}%
-\frac{1}{2}\left(  d_{1}+d_{2}+d_{3}\right)  ^{2}\left[  X+Y,\left[
Y,X\right]  \right]  -\\
\frac{1}{4}\left(  d_{1}+d_{2}+d_{3}\right)  ^{3}\left[  X+Y,\left[  \left[
X,Y\right]  ,X\right]  \right]  -\\
\frac{1}{6}\left(  d_{1}+d_{2}+d_{3}\right)  ^{3}\left[  X+Y,\left[
X+Y,\left[  Y,X\right]  \right]  \right]
\end{array}
\right)  +\\
\frac{1}{12}\left(  d_{1}+d_{2}+d_{3}\right)  ^{3}\left[  X+Y,\left[
X+Y,\left[  Y,X\right]  \right]  \right]
\end{array}
\right\}  .\\
&  \exp\,\left(  d_{1}+d_{2}+d_{3}\right)  \left(  X+Y\right)  +\frac{1}%
{2}\left(  d_{1}+d_{2}+d_{3}\right)  ^{2}\left[  X,Y\right]  +\\
&  \frac{1}{12}\left(  d_{1}+d_{2}+d_{3}\right)  ^{3}\left[  X-Y,\left[
X,Y\right]  \right]  +\\
&  d_{4}\left\{
\begin{array}
[c]{c}%
X+Y+\frac{1}{2}\left(  d_{1}+d_{2}+d_{3}\right)  \left[  X,Y\right]  +\frac
{1}{4}\left(  d_{1}+d_{2}+d_{3}\right)  ^{2}\left[  \left[  X,Y\right]
,Y\right]  -\\
\frac{1}{6}\left(  d_{1}+d_{2}+d_{3}\right)  ^{2}\left[  X+Y,\left[
X,Y\right]  \right]  +\\
\frac{1}{4}\left(  d_{1}+d_{2}+d_{3}\right)  ^{2}\left[  X+Y,\left[
X,Y\right]  \right]  -\\
\frac{1}{2}\left(  d_{1}+d_{2}+d_{3}\right)  \left[  Y,X\right]  -\frac{1}%
{4}\left(  d_{1}+d_{2}+d_{3}\right)  ^{2}\left[  \left[  X,Y\right]
,X\right]  -\\
\frac{1}{6}\left(  d_{1}+d_{2}+d_{3}\right)  ^{2}\left[  X+Y,\left[
Y,X\right]  \right]
\end{array}
\right\}  .\\
&  \exp\,d_{4}\left\{
\begin{array}
[c]{c}%
\frac{1}{24}\left(  d_{1}+d_{2}+d_{3}\right)  ^{3}\left[  \left[  X-Y,\left[
X,Y\right]  \right]  ,Y\right]  -\\
\frac{1}{12}\left(  d_{1}+d_{2}+d_{3}\right)  ^{3}\left[  X+Y,\left[  \left[
X,Y\right]  ,Y\right]  \right]  +\\
\frac{1}{2}\left(
\begin{array}
[c]{c}%
\frac{1}{4}\left(  d_{1}+d_{2}+d_{3}\right)  ^{3}\left[  X+Y,\left[  \left[
X,Y\right]  ,Y\right]  \right]  -\\
\frac{1}{6}\left(  d_{1}+d_{2}+d_{3}\right)  ^{3}\left[  X+Y,\left[
X+Y,\left[  X,Y\right]  \right]  \right]
\end{array}
\right)  -\\
\frac{1}{12}\left(  d_{1}+d_{2}+d_{3}\right)  ^{3}\left[  X+Y,\left[
X+Y,\left[  X,Y\right]  \right]  \right]  +\\
\frac{1}{8}\left(  d_{1}+d_{2}+d_{3}\right)  ^{3}\left[  X+Y,\left[
X+Y,\left[  X,Y\right]  \right]  \right]
\end{array}
\right\} \\
&  \left)
\begin{array}
[c]{c}%
\text{By Theorem \ref{t5.6}\ with }\\
\delta^{\mathrm{rightt}}\left(  \exp\right)  \left(
\begin{array}
[c]{c}%
\left(  d_{1}+d_{2}+d_{3}\right)  \left(  X+Y\right)  +\\
\frac{1}{2}\left(  d_{1}+d_{2}+d_{3}\right)  ^{2}\left[  X,Y\right]  +\\
\frac{1}{12}\left(  d_{1}+d_{2}+d_{3}\right)  ^{3}\left[  X-Y,\left[
X,Y\right]  \right]  +\\
d_{4}\left\{
\begin{array}
[c]{c}%
X+Y+\frac{1}{2}\left(  d_{1}+d_{2}+d_{3}\right)  \left[  X,Y\right]  +\\
\frac{1}{4}\left(  d_{1}+d_{2}+d_{3}\right)  ^{2}\left[  \left[  X,Y\right]
,Y\right]  -\\
\frac{1}{6}\left(  d_{1}+d_{2}+d_{3}\right)  ^{2}\left[  X+Y,\left[
X,Y\right]  \right]  +\\
\frac{1}{4}\left(  d_{1}+d_{2}+d_{3}\right)  ^{2}\left[  X+Y,\left[
X,Y\right]  \right]
\end{array}
\right\}
\end{array}
\right) \\
\left(
\begin{array}
[c]{c}%
-\frac{1}{2}\left(  d_{1}+d_{2}+d_{3}\right)  \left[  Y,X\right]  -\frac{1}%
{4}\left(  d_{1}+d_{2}+d_{3}\right)  ^{2}\left[  \left[  X,Y\right]
,X\right]  -\\
\frac{1}{6}\left(  d_{1}+d_{2}+d_{3}\right)  ^{2}\left[  X+Y,\left[
Y,X\right]  \right]
\end{array}
\right) \\
=-\frac{1}{2}\left(  d_{1}+d_{2}+d_{3}\right)  \left[  Y,X\right]  -\frac
{1}{4}\left(  d_{1}+d_{2}+d_{3}\right)  ^{2}\left[  \left[  X,Y\right]
,X\right]  -\\
\frac{1}{6}\left(  d_{1}+d_{2}+d_{3}\right)  ^{2}\left[  X+Y,\left[
Y,X\right]  \right]  +\\
\frac{1}{2}\left(
\begin{array}
[c]{c}%
-\frac{1}{2}\left(  d_{1}+d_{2}+d_{3}\right)  ^{2}\left[  X+Y,\left[
Y,X\right]  \right]  -\\
\frac{1}{4}\left(  d_{1}+d_{2}+d_{3}\right)  ^{3}\left[  X+Y,\left[  \left[
X,Y\right]  ,X\right]  \right]  -\\
\frac{1}{6}\left(  d_{1}+d_{2}+d_{3}\right)  ^{3}\left[  X+Y,\left[
X+Y,\left[  Y,X\right]  \right]  \right]
\end{array}
\right)  -\\
\frac{1}{12}\left(  d_{1}+d_{2}+d_{3}\right)  ^{3}\left[  X+Y,\left[
X+Y,\left[  Y,X\right]  \right]  \right]
\end{array}
\right(
\end{align*}
We keep on once more.
\begin{align*}
&  =\exp\,d_{4}\left\{
\begin{array}
[c]{c}%
-\frac{1}{24}\left(  d_{1}+d_{2}+d_{3}\right)  ^{3}\left[  \left[  X-Y,\left[
X,Y\right]  \right]  ,X\right]  -\\
\frac{1}{12}\left(  d_{1}+d_{2}+d_{3}\right)  ^{3}\left[  X+Y,\left[  \left[
X,Y\right]  ,X\right]  \right]  -\\
\frac{1}{2}\left(
\begin{array}
[c]{c}%
-\frac{1}{4}\left(  d_{1}+d_{2}+d_{3}\right)  ^{3}\left[  X+Y,\left[  \left[
X,Y\right]  ,X\right]  \right]  -\\
\frac{1}{6}\left(  d_{1}+d_{2}+d_{3}\right)  ^{3}\left[  X+Y,\left[
X+Y,\left[  Y,X\right]  \right]  \right]
\end{array}
\right)  +\\
\frac{1}{12}\left(  d_{1}+d_{2}+d_{3}\right)  ^{3}\left[  X+Y,\left[
X+Y,\left[  Y,X\right]  \right]  \right]  -\\
\frac{1}{8}\left(  d_{1}+d_{2}+d_{3}\right)  ^{3}\left[  X+Y,\left[
X+Y,\left[  Y,X\right]  \right]  \right]
\end{array}
\right\}  .\\
&  \exp\,d_{4}\left\{
\begin{array}
[c]{c}%
\frac{1}{4}\left(  d_{1}+d_{2}+d_{3}\right)  ^{2}\left[  X+Y,\left[
Y,X\right]  \right]  +\\
\frac{1}{8}\left(  d_{1}+d_{2}+d_{3}\right)  ^{3}\left[  X+Y,\left[
X+Y,\left[  Y,X\right]  \right]  \right]
\end{array}
\right\}  .\\
&  \exp\,\left(  d_{1}+d_{2}+d_{3}\right)  \left(  X+Y\right)  +\frac{1}%
{2}\left(  d_{1}+d_{2}+d_{3}\right)  ^{2}\left[  X,Y\right]  +\\
&  \frac{1}{12}\left(  d_{1}+d_{2}+d_{3}\right)  ^{3}\left[  X-Y,\left[
X,Y\right]  \right]  +\\
&  d_{4}\left\{
\begin{array}
[c]{c}%
X+Y+\frac{1}{2}\left(  d_{1}+d_{2}+d_{3}\right)  \left[  X,Y\right]  +\frac
{1}{4}\left(  d_{1}+d_{2}+d_{3}\right)  ^{2}\left[  \left[  X,Y\right]
,Y\right]  -\\
\frac{1}{6}\left(  d_{1}+d_{2}+d_{3}\right)  ^{2}\left[  X+Y,\left[
X,Y\right]  \right]  +\\
\frac{1}{4}\left(  d_{1}+d_{2}+d_{3}\right)  ^{2}\left[  X+Y,\left[
X,Y\right]  \right]  -\\
\frac{1}{2}\left(  d_{1}+d_{2}+d_{3}\right)  \left[  Y,X\right]  -\frac{1}%
{4}\left(  d_{1}+d_{2}+d_{3}\right)  ^{2}\left[  \left[  X,Y\right]
,X\right]  -\\
\frac{1}{6}\left(  d_{1}+d_{2}+d_{3}\right)  ^{2}\left[  X+Y,\left[
Y,X\right]  \right]
\end{array}
\right\}  .\\
&  \exp\,d_{4}\left\{
\begin{array}
[c]{c}%
\frac{1}{24}\left(  d_{1}+d_{2}+d_{3}\right)  ^{3}\left[  \left[  X-Y,\left[
X,Y\right]  \right]  ,Y\right]  -\\
\frac{1}{12}\left(  d_{1}+d_{2}+d_{3}\right)  ^{3}\left[  X+Y,\left[  \left[
X,Y\right]  ,Y\right]  \right]  +\\
\frac{1}{2}\left(
\begin{array}
[c]{c}%
\frac{1}{4}\left(  d_{1}+d_{2}+d_{3}\right)  ^{3}\left[  X+Y,\left[  \left[
X,Y\right]  ,Y\right]  \right]  -\\
\frac{1}{6}\left(  d_{1}+d_{2}+d_{3}\right)  ^{3}\left[  X+Y,\left[
X+Y,\left[  X,Y\right]  \right]  \right]
\end{array}
\right)  -\\
\frac{1}{12}\left(  d_{1}+d_{2}+d_{3}\right)  ^{3}\left[  X+Y,\left[
X+Y,\left[  X,Y\right]  \right]  \right]  +\\
\frac{1}{8}\left(  d_{1}+d_{2}+d_{3}\right)  ^{3}\left[  X+Y,\left[
X+Y,\left[  X,Y\right]  \right]  \right]
\end{array}
\right\}
\end{align*}
We keep on once more.
\begin{align*}
&  =\exp\,d_{4}\left\{
\begin{array}
[c]{c}%
-\frac{1}{24}\left(  d_{1}+d_{2}+d_{3}\right)  ^{3}\left[  \left[  X-Y,\left[
X,Y\right]  \right]  ,X\right]  -\\
\frac{1}{12}\left(  d_{1}+d_{2}+d_{3}\right)  ^{3}\left[  X+Y,\left[  \left[
X,Y\right]  ,X\right]  \right]  -\\
\frac{1}{2}\left(
\begin{array}
[c]{c}%
-\frac{1}{4}\left(  d_{1}+d_{2}+d_{3}\right)  ^{3}\left[  X+Y,\left[  \left[
X,Y\right]  ,X\right]  \right]  -\\
\frac{1}{6}\left(  d_{1}+d_{2}+d_{3}\right)  ^{3}\left[  X+Y,\left[
X+Y,\left[  Y,X\right]  \right]  \right]
\end{array}
\right)  +\\
\frac{1}{12}\left(  d_{1}+d_{2}+d_{3}\right)  ^{3}\left[  X+Y,\left[
X+Y,\left[  Y,X\right]  \right]  \right]  -\\
\frac{1}{8}\left(  d_{1}+d_{2}+d_{3}\right)  ^{3}\left[  X+Y,\left[
X+Y,\left[  Y,X\right]  \right]  \right]
\end{array}
\right\}  .\\
&  \exp\,\left(  d_{1}+d_{2}+d_{3}\right)  \left(  X+Y\right)  +\frac{1}%
{2}\left(  d_{1}+d_{2}+d_{3}\right)  ^{2}\left[  X,Y\right]  +\\
&  \frac{1}{12}\left(  d_{1}+d_{2}+d_{3}\right)  ^{3}\left[  X-Y,\left[
X,Y\right]  \right]  +\\
&  d_{4}\left\{
\begin{array}
[c]{c}%
X+Y+\frac{1}{2}\left(  d_{1}+d_{2}+d_{3}\right)  \left[  X,Y\right]  +\frac
{1}{4}\left(  d_{1}+d_{2}+d_{3}\right)  ^{2}\left[  \left[  X,Y\right]
,Y\right]  -\\
\frac{1}{6}\left(  d_{1}+d_{2}+d_{3}\right)  ^{2}\left[  X+Y,\left[
X,Y\right]  \right]  +\\
\frac{1}{4}\left(  d_{1}+d_{2}+d_{3}\right)  ^{2}\left[  X+Y,\left[
X,Y\right]  \right]  -\\
\frac{1}{2}\left(  d_{1}+d_{2}+d_{3}\right)  \left[  Y,X\right]  -\frac{1}%
{4}\left(  d_{1}+d_{2}+d_{3}\right)  ^{2}\left[  \left[  X,Y\right]
,X\right]  -\\
\frac{1}{6}\left(  d_{1}+d_{2}+d_{3}\right)  ^{2}\left[  X+Y,\left[
Y,X\right]  \right]  +\\
\frac{1}{4}\left(  d_{1}+d_{2}+d_{3}\right)  ^{2}\left[  X+Y,\left[
Y,X\right]  \right]
\end{array}
\right\}  .\\
&  \exp\,d_{4}\left\{
\begin{array}
[c]{c}%
\frac{1}{24}\left(  d_{1}+d_{2}+d_{3}\right)  ^{3}\left[  \left[  X-Y,\left[
X,Y\right]  \right]  ,Y\right]  -\\
\frac{1}{12}\left(  d_{1}+d_{2}+d_{3}\right)  ^{3}\left[  X+Y,\left[  \left[
X,Y\right]  ,Y\right]  \right]  +\\
\frac{1}{2}\left(
\begin{array}
[c]{c}%
\frac{1}{4}\left(  d_{1}+d_{2}+d_{3}\right)  ^{3}\left[  X+Y,\left[  \left[
X,Y\right]  ,Y\right]  \right]  -\\
\frac{1}{6}\left(  d_{1}+d_{2}+d_{3}\right)  ^{3}\left[  X+Y,\left[
X+Y,\left[  X,Y\right]  \right]  \right]
\end{array}
\right)  -\\
\frac{1}{12}\left(  d_{1}+d_{2}+d_{3}\right)  ^{3}\left[  X+Y,\left[
X+Y,\left[  X,Y\right]  \right]  \right]  +\\
\frac{1}{8}\left(  d_{1}+d_{2}+d_{3}\right)  ^{3}\left[  X+Y,\left[
X+Y,\left[  X,Y\right]  \right]  \right]
\end{array}
\right\} \\
&  \left)
\begin{array}
[c]{c}%
\text{By Theorem \ref{t5.6}\ with }\\
\delta^{\mathrm{rightt}}\left(  \exp\right)  \left(
\begin{array}
[c]{c}%
\left(  d_{1}+d_{2}+d_{3}\right)  \left(  X+Y\right)  +\\
\frac{1}{2}\left(  d_{1}+d_{2}+d_{3}\right)  ^{2}\left[  X,Y\right]  +\\
\frac{1}{12}\left(  d_{1}+d_{2}+d_{3}\right)  ^{3}\left[  X-Y,\left[
X,Y\right]  \right]  +\\
d_{4}\left\{
\begin{array}
[c]{c}%
X+Y+\frac{1}{2}\left(  d_{1}+d_{2}+d_{3}\right)  \left[  X,Y\right]  +\\
\frac{1}{4}\left(  d_{1}+d_{2}+d_{3}\right)  ^{2}\left[  \left[  X,Y\right]
,Y\right]  -\\
\frac{1}{6}\left(  d_{1}+d_{2}+d_{3}\right)  ^{2}\left[  X+Y,\left[
X,Y\right]  \right]  +\\
\frac{1}{4}\left(  d_{1}+d_{2}+d_{3}\right)  ^{2}\left[  X+Y,\left[
X,Y\right]  \right]  -\\
\frac{1}{2}\left(  d_{1}+d_{2}+d_{3}\right)  \left[  Y,X\right]  -\\
\frac{1}{4}\left(  d_{1}+d_{2}+d_{3}\right)  ^{2}\left[  \left[  X,Y\right]
,X\right]  -\\
\frac{1}{6}\left(  d_{1}+d_{2}+d_{3}\right)  ^{2}\left[  X+Y,\left[
Y,X\right]  \right]
\end{array}
\right\}
\end{array}
\right) \\
\left(  \frac{1}{4}\left(  d_{1}+d_{2}+d_{3}\right)  ^{2}\left[  X+Y,\left[
Y,X\right]  \right]  \right) \\
=\frac{1}{4}\left(  d_{1}+d_{2}+d_{3}\right)  ^{2}\left[  X+Y,\left[
Y,X\right]  \right]  +\\
\frac{1}{8}\left(  d_{1}+d_{2}+d_{3}\right)  ^{3}\left[  X+Y,\left[
X+Y,\left[  Y,X\right]  \right]  \right]
\end{array}
\right(
\end{align*}

We keep on once more.
\begin{align*}
&  =\exp\,\left(  d_{1}+d_{2}+d_{3}\right)  \left(  X+Y\right)  +\frac{1}%
{2}\left(  d_{1}+d_{2}+d_{3}\right)  ^{2}\left[  X,Y\right]  +\\
&  \frac{1}{12}\left(  d_{1}+d_{2}+d_{3}\right)  ^{3}\left[  X-Y,\left[
X,Y\right]  \right]  +\\
&  d_{4}\left\{  X+Y+\left(  d_{1}+d_{2}+d_{3}\right)  \left[  X,Y\right]
+\frac{1}{4}\left(  d_{1}+d_{2}+d_{3}\right)  ^{2}\left[  X-Y,\left[
X,Y\right]  \right]  \right\}  .\\
&  \exp\,d_{4}\left(  d_{1}+d_{2}+d_{3}\right)  ^{3}\left\{
\begin{array}
[c]{c}%
\frac{1}{24}\left[  X,\left[  X-Y,\left[  X,Y\right]  \right]  \right]  +\\
\left(  \frac{1}{12}-\frac{1}{8}\right)  \left[  X+Y,\left[  X,\left[
X,Y\right]  \right]  \right]  +\\
\left(  \frac{1}{12}+\frac{1}{12}-\frac{1}{8}\right)  \left[  X+Y,\left[
X+Y,\left[  Y,X\right]  \right]  \right]  -\\
\frac{1}{24}\left[  Y,\left[  X-Y,\left[  X,Y\right]  \right]  \right]  +\\
\left(  \frac{1}{12}-\frac{1}{8}\right)  \left[  X+Y,\left[  Y,\left[
X,Y\right]  \right]  \right]  -\\
\left(  \frac{1}{12}+\frac{1}{12}-\frac{1}{8}\right)  \left[  X+Y,\left[
X+Y,\left[  X,Y\right]  \right]  \right]
\end{array}
\right\} \\
&  =\exp\,\left(  d_{1}+d_{2}+d_{3}+d_{4}\right)  \left(  X+Y\right)
+\frac{1}{2}\left(  d_{1}+d_{2}+d_{3}+d_{4}\right)  ^{2}\left[  X,Y\right]
+\\
&  \frac{1}{12}\left(  d_{1}+d_{2}+d_{3}+d_{4}\right)  ^{3}\left[  X-Y,\left[
X,Y\right]  \right]  .\\
&  \exp\,d_{4}\left(  d_{1}+d_{2}+d_{3}\right)  ^{3}\left\{
\begin{array}
[c]{c}%
-\frac{1}{24}\left(  \left[  X,\left[  Y,\left[  X,Y\right]  \right]  \right]
+\left[  Y,\left[  X,\left[  X,Y\right]  \right]  \right]  \right)  +\\
\frac{1}{24}\left[  X+Y,\left[  X+Y,\left[  Y,X\right]  \right]  \right]  -\\
-\frac{1}{24}\left(  \left[  X,\left[  Y,\left[  X,Y\right]  \right]  \right]
+\left[  Y,\left[  X,\left[  X,Y\right]  \right]  \right]  \right)  -\\
\frac{1}{24}\left[  X+Y,\left[  X+Y,\left[  X,Y\right]  \right]  \right]
\end{array}
\right\} \\
&  =\exp\,\left(  d_{1}+d_{2}+d_{3}+d_{4}\right)  \left(  X+Y\right)
+\frac{1}{2}\left(  d_{1}+d_{2}+d_{3}+d_{4}\right)  ^{2}\left[  X,Y\right]
+\\
&  \frac{1}{12}\left(  d_{1}+d_{2}+d_{3}+d_{4}\right)  ^{3}\left[  X-Y,\left[
X,Y\right]  \right]  .\\
&  \exp\,-\frac{1}{12}d_{4}\left(  d_{1}+d_{2}+d_{3}\right)  ^{3}\left(
\begin{array}
[c]{c}%
\left[  X,\left[  Y,\left[  X,Y\right]  \right]  \right]  +\left[  Y,\left[
X,\left[  X,Y\right]  \right]  \right]  +\\
\left[  X+Y,\left[  X+Y,\left[  X,Y\right]  \right]  \right]
\end{array}
\right) \\
&  =\exp\,\left(  d_{1}+d_{2}+d_{3}+d_{4}\right)  \left(  X+Y\right)
+\frac{1}{2}\left(  d_{1}+d_{2}+d_{3}+d_{4}\right)  ^{2}\left[  X,Y\right]
+\\
&  \frac{1}{12}\left(  d_{1}+d_{2}+d_{3}+d_{4}\right)  ^{3}\left[  X-Y,\left[
X,Y\right]  \right]  -\\
&  \frac{1}{48}\left(  d_{1}+d_{2}+d_{3}+d_{4}\right)  ^{4}\left(
\begin{array}
[c]{c}%
\left[  X,\left[  Y,\left[  X,Y\right]  \right]  \right]  +\left[  Y,\left[
X,\left[  X,Y\right]  \right]  \right]  +\\
\left[  X+Y,\left[  X+Y,\left[  X,Y\right]  \right]  \right]
\end{array}
\right)
\end{align*}

\end{proof}

\end{document}